\documentclass[a4paper]{article}       

\usepackage[pdftex]{graphicx}	
\usepackage{bmpsize}
\usepackage{amsmath, amssymb, amsthm, amsfonts}
\usepackage{amsmath, amsfonts}
\usepackage{ascmac, fancybox}
\usepackage{algorithm, algorithmic}
\usepackage{mathrsfs}	
\usepackage{braket}		
\usepackage{float}		
\usepackage{geometry}
\usepackage{enumerate}

\newtheorem{theorem}    {Theorem}[section]
\newtheorem{lemma}      [theorem]{Lemma}

\newtheorem{proposition}        [theorem]{Proposition}

\newtheorem{example}    	[theorem]{Example}

\newcommand{\calP}[1][\null]{
	\if #1\null%
		\mathcal{P}%
	\else{ \mathcal{P}_{#1} }
	\fi
}

\newcommand{\fish}{g_F}
\newcommand{\connm}{\nabla^{(m)}}
\newcommand{\conne}{\nabla^{(e)}}
\newcommand{\KL}[2]{KL\left[ #1 \| #2 \right]}
\newcommand{\onevec}[1]{1_{#1}}

\newcommand{\lam}{\lambda}
\newcommand{\etij}{\eta_{ij}}
\newcommand{\etim}{\eta_{im}}
\newcommand{\etin}{\eta_{in}}
\newcommand{\etnj}{\eta_{nj}}

\newcommand{\thij}{\theta^{ij}}
\newcommand{\thim}{\theta^{im}}
\newcommand{\thin}{\theta^{in}}
\newcommand{\thnj}{\theta^{nj}}

\newcommand{\Etij}[1][k]{H^{#1}_{ij}}
\newcommand{\Etin}[1][k]{H^{#1}_{in}}
\newcommand{\Etnj}[1][k]{H^{#1}_{nj}}
\newcommand{\Thij}[1][k]{\Theta_{#1}^{ij}}
\newcommand{\Thin}[1][k]{\Theta_{#1}^{in}}
\newcommand{\Thnj}[1][k]{\Theta_{#1}^{nj}}
\newcommand{\Frechet}{Fr\'{e}chet}

\newcommand{\tA}{\tilde{A}}

\newcommand{\tH}{\tilde{\mathcal{H}}}
\newcommand{\tT}{\tilde{\mathcal{T}}}
\newcommand{\tPhi}{\tilde{\Phi}}

\newcommand{\tq}{\tilde{q}}

\newcommand{\hPhi}{\hat{\Phi}}

\newcommand{\cPhi}{\check{\Phi}}

\newcommand{\cg}{\check{g}}
\newcommand{\cdel}{\check{\nabla}}

\newcommand{\del}{\partial}
\newcommand{\p}{\partial}
\newcommand{\hs}{\hspace{10mm}}
\newcommand{\etal}[1][~]{\textit{et al.}#1}

\newcommand{\Rnm}{\R_{++}^{n\times m}}
\newcommand{\Pnm}{\calP_{nm-1}}
\newcommand{\Pnn}{\calP_{n^2-1}}

\newcommand{\Pt}{P^{(t)}}
\newcommand{\Ptt}{P^{(t+1)}}
\newcommand{\Qt}{Q^{(t)}}

\newcommand{\ut}{u^{(t)}}
\newcommand{\utt}{u^{(t+1)}}
\newcommand{\vt}{v^{(t)}}
\newcommand{\vtt}{v^{(t+1)}}
\newcommand{\at}{\alpha^{(t)}}
\newcommand{\att}{\alpha^{(t+1)}}
\newcommand{\bt}{\beta^{(t)}}
\newcommand{\btt}{\beta^{(t+1)}}

\newcommand{\PtN}[1]{P^{(#1;t)}}
\newcommand{\PttN}[1]{P^{(#1;t+1)}}
\newcommand{\QtN}[1]{Q^{(#1;t)}}
\newcommand{\atN}[1]{\alpha^{(#1;t)}}
\newcommand{\attN}[1]{\alpha^{(#1;t+1)}}
\newcommand{\btN}[1]{\beta^{(#1;t)}}
\newcommand{\bttN}[1]{\beta^{(#1;t+1)}}

\newcommand{\dM}{\mathbb{M}}

\newcommand{\ep}{\varepsilon}
\newcommand{\vphi}{\varphi}
\newcommand{\N}{\mathbb{N}}
\newcommand{\R}{\mathbb{R}}  
\newcommand{\calH}{\mathcal{H}}
\newcommand{\calT}{\mathcal{T}}

\newcommand{\inc}{\subset}
\newcommand{\any}{\null^{\forall}}
\newcommand{\exs}{\null^{\exists}}
\newcommand{\st}{\text{ s.t. }}
\newcommand{\ow}{\text{otherwise}}
\newcommand{\1}[1]{\frac{1}{#1}}
\newcommand{\pd}[2]{\frac{\partial #1}{\partial #2}}

\newcommand{\sumn}[2][n]{\sum_{#2 = 1}^{#1}}
\newcommand{\ran}[3][\null]{1\leq #1#2 \leq #3}
\renewcommand{\(}{\left(}
\renewcommand{\)}{\right)}
\newcommand{\<}{\left\langle}	
\renewcommand{\>}{\right\rangle}

\begin{document}

\title{
	Optimal transport problems regularized by generic convex functions: 
	A geometric and algorithmic approach
}
\author{
	Daiji Tsutsui \\[2mm]
	\normalsize Department of Mathematics, Osaka University, \\
	\normalsize Toyonaka, Osaka 560-0043, Japan \\
	\normalsize d-tsutsui@cr.math.sci.osaka-u.ac.jp 
}
\date{}
\maketitle

\begin{abstract}
	In order to circumvent the difficulties in solving numerically the discrete optimal transport problem, 
	in which one minimizes the linear target function $P\mapsto\<C,P\>:=\sum_{i,j}C_{ij}P_{ij}$, 
	Cuturi introduced a variant of the problem in which 
	the target function is altered by a convex one $\Phi(P)=\<C,P\>-\lam\calH(P)$, 
	where $\calH$ is the Shannon entropy and $\lam$ is a positive constant. 
	We herein generalize their formulation to a target function
	of the form $\Phi(P)=\<C,P\>+\lam f(P)$, 
	where $f$ is a generic strictly convex smooth function. 
	We also propose an iterative method for finding a numerical solution,  
	and clarify that the proposed method is particularly efficient when $f(P)=\1{2}\|P\|^2$. 
\end{abstract}


	\section{Introduction}
	
The optimal transportation theory \cite{Villani03,Santambrogio15} was established to solve 
a following type of a minimization problem. 
Let us consider a discrete setting, that is, the sets of sources and targets are finite. 
Let $\calP_{n-1}$ be the set of nonsingular probability distributions on $\{1,2,\dots,n\}$, 
which is given by
\begin{align*}
	\calP_{n-1} := \Set{p\in\R^n | \sumn[n]{i} p_i = 1, ~p_i > 0, 1\leq i\leq n},  
\end{align*}
where the subscript $n-1$ describes the dimension as a manifold.
We regard each element of $\calP_{n-1}$ as a distribution of resources positioned
on the set of $n$ sources. 
We also assign $\calP_{m-1}$ to the set of distributions of resources required in $m$ targets. 
In order to represent transportation,
we introduce a set $\Pnm$ defined by
\begin{align*}
	\Pnm := \Set{P=(P_{ij})\in\R^{n\times m} | \sum_{i,j} P_{ij} = 1, ~P_{ij} > 0, 1\leq i\leq n, 1\leq j\leq m}.     
\end{align*}
Each entry $P_{ij}$ means a quantity transported from the source $i$ to the target $j$, 
and hence, each $P\in\Pnm$ is called a \textit{transport plan}. 
Given $p\in\calP_{n-1},q\in\calP_{m-1}$, each element of the subset 
\begin{align*}
	\Pi(p,q) := \Set{P\in\Pnm | \sumn{j}P_{ij} = p_i, ~ \sumn{i}P_{ij} = q_j}
\end{align*}
is called a transport plan from $p$ to $q$ or a \textit{coupling} of $p$ and $q$. 
Let $C=(C^{ij})$ be a given matrix
whose entry $C^{ij}$ means a transport cost from $i$ to $j$.  
Then, the \textit{optimal transport problem} is presented as the minimization problem: 
\begin{align}	\label{eq:Was-prob}
	W(p,q) := \inf_{P\in\Pi(p,q)} \<P,C\> = \inf_{P\in\Pi(p,q)} \sum_{i,j}P_{ij}C^{ij},  
\end{align}
where the target function $\<P,C\>$ is the total cost of the transport plan $P$. 
When $m=n$ and $C$ is a distance matrix, 
the minimum $W(p,q)$ defines a distance on $\calP_{n-1}$ 
that is often referred to as the \textit{Wasserstein distance}. 

The optimal transport problem \eqref{eq:Was-prob} can be solved by the linear programming; 
however, it costs $O(\max\{n,m\}^3)$ time. 
Cuturi proposed an alternative minimization problem
by changing the target function to a convex one, 
which is usually referred to as the \textit{entropic regularization} \cite{Cuturi13}. 
He introduced the problem: 
\begin{align}	\label{eq:Cut-prob}
	W_\lambda(p,q) &:= \inf_{P\in\Pi(p,q)} \<C,P\> - \lambda \calH(P), \quad \lambda>0,   \\
	\calH(P) &:= -\sum_{ij} P_{ij} \log P_{ij}.  \notag
\end{align}
Here, the function $\calH$ is known as the Shannon entropy
and is smooth and concave with respect to the ordinary affine structure of $\Pnm$. 
While the quantity $W_\lam(p,q)$, 
in contrast with $W(p,q)$, 
does not define a distance on $\calP_{n-1}$, 
$W_\lam(p,q)$ gives an approximation of the Wasserstein distance  
in the sense that $W_\lam(p,q)$ converges to $W(p,q)$ as $\lam$ tends to $0$ \cite{CuturiPeyre16}. 
Cuturi showed that $W_\lam(p,q)$ is obtained by the \textit{Sinkhorn algorithm}, 
which costs $O(\max\{n,m\}^2)$ time \cite{Cuturi13}. 
When $m$ or $n$ is large, one can thus compute $W_\lam(p,q)$ much faster than $W(p,q)$.  

Cuturi's entropic regularization has a wide range of applications. 
Frogner \etal \cite{FrognerZMAP} used $W_\lam$ as a loss function for a supervised learning. 
Courty \etal \cite{CourtyFTR} applied it to the unsupervised domain adaptation, 
which is a type of classification. 
In the field of image processing, as another application, 
a method for interpolating several patterns in the metric space $(\calP_{n-1}, W)$ has been devised.  
It is mathematically given as a solution of the minimization problem: 
\begin{align*}
	\textbf{Minimize} ~\sumn[N]{k}r_k W(p^k,q) ~\mathrm{under} ~q\in\calP_{n-1},  
\end{align*}
where $p^1,\dots, p^N\in\calP$, $r_1,\dots,r_N>0$ with $\sum_{k=1}^N r_k = 1$, 
and is called the Wasserstein barycenter \cite{AguehCarlier}. 
Also in this problem, 
by replacing $W$ to $W_\lam$, an approximation of the barycenter can be computed by
a high-speed algorithm \cite{CuturiDoucet,BenamouCCNP}. 

Amari \etal \cite{AmariKO} pointed out that the set $\calP^{opt}$ of all optimal plans 
in Cuturi's problem \eqref{eq:Cut-prob} is an exponential family, 
and clarified that the entropic regularization is described in terms of the information geometry \cite{AmariNagaoka}. 
Their work illustrated that the Sinkhorn algorithm searches for a solution
by successive applications of the $e$-projection to a pair $m$-autoparallel submanifolds. 
Due to the projection theorem in information geometry, 
the Kullback-Leibler divergence measured from the solution decreases monotonically. 
A similar geometric interpretation for the Wasserstein barycenter was given by Benamou \etal \cite{BenamouCCNP}.

Muzellec \etal \cite{MuzellecNPN} extended Amari-Cuturi's framework
to the optimal transport problem regularized by the Tsallis entropy. 
Essid and Solomon \cite{EssidSolomon} studied the transport problem with quadratic regularization on a graph, 
and established a Newton-type algorithm to solve that.  


In this paper, we generalize Amari-Cuturi's framework to the problem
\begin{align}	\label{eq:gen-conv-prob}
	\vphi(p,q) &:= \inf_{P\in\Pi(p,q)} \<P,C\> + \lam f(P) , 
\end{align}
where $f:\Pnm\to\R$ is a smooth convex function and $\lambda$ is a positive constant. 
Denoting by $\Phi(P)$, the target function $P\mapsto\<P,C\> + \lam f(P)$, 
this problem is rewritten as
\begin{align}	\label{eq:gen-prob}
	\vphi(p,q) &= \inf_{P\in\Pi(p,q)} \Phi(P),   
\end{align}
for a smooth convex function $\Phi$ on $\Pnm$. 
We further devise a procedure to find a solution of \eqref{eq:gen-prob}
and the relaxed barycenter problem: 
\begin{align}	\label{eq:gen-bary}
	\textbf{Minimize} ~\sumn[N]{k}r_k \, \vphi(p^k,q) ~\mathrm{under} ~q\in\calP_{n-1}.      
\end{align}
When the domain of a dual problem has no restriction,
it is proved that our procedure monotonically decreases 
the Bregman divergence \cite{Bregman} as measured from the optimal solution.  
We also address when the dual domain is unbounded from above, 
and we establish generalizations of the algorithms
based on the non-smooth convex analysis. 
In numerical simulations, we show that the generalized algorithms works particularly well
in the quadratic regularized optimal transport problem, 
which is the problem \eqref{eq:gen-conv-prob} with $f(P)=\1{2}\|P\|^2$. 

This paper is organized as follows. 
Section 2 gives a brief review on Amari-Cuturi's framework for the entropic regularization. 
In Section 3, we investigate the problems \eqref{eq:gen-prob} and \eqref{eq:gen-bary} 
from an information geometric point of view
under the assumption that the dual domain has no restriction.  
In Section 4, relaxing the assumption, 
we study those problems when the dual domain is unbounded from the above. 
We also confirm that our algorithm works efficiently when $f(P)=\1{2}\|P\|^2$ in numerical simulations.  
Section 5 offers concluding remarks.


	\section{Entropic regularization of optimal transport problem}
	

\subsection{Cuturi's entropic regularization} 
We herein outline the entropic regularization of the optimal transport problem by Cuturi. 
In his article \cite{Cuturi13}, 
instead of the optimization problem \eqref{eq:Was-prob},  
Cuturi considered the alternative problem \eqref{eq:Cut-prob}. 
Under the natural affine structure of $\Pnm$, 
the target function 
\begin{equation}	\label{eq:phi-lambda}
	\Phi_\lam(P) := \<C,P\> - \lambda \calH(P)
\end{equation}
in Cuturi's problem \eqref{eq:Cut-prob} is strictly convex 
in contrast with that of the original problem \eqref{eq:Was-prob} because of the Shannon entropy $\calH$. 
Since the function $\Phi_\lam(P)$ is strictly convex on a convex affine subspace $\Pi(p,q)$, 
the problem \eqref{eq:Cut-prob} has a unique optimum $P^*(p,q)$, 
which is called the \textit{optimal transport plan} from $p$ to $q$,  
at least on the closure $\overline{\Pi(p,q)}\inc\R^{n\times m}$. 

Cuturi showed that there exists a solution $(\alpha, \beta)\in\R^{n+m}$ of a dual problem \cite{Cuturi13}, 
which gives the optimum as 
\begin{equation}	\label{eq:optimal-trans}
	P^*(p,q)_{ij} = \exp\( \1{\lam}(\alpha_i + \beta_j - C^{ij}) \). 
\end{equation}
Due to the form \eqref{eq:optimal-trans} of the optimal plan, 
Cuturi proposed that one can compute $P^*(p,q)$ at high speed 
by the Sinkhorn algorithm presented in Algorithm \ref{alg1}. 
Putting $K_{ij}:=\exp(-C^{ij}/\lam), u_i=\exp(\alpha_i/\lam), v_j=\exp(\beta_j/\lam)$, 
it is also denoted as 
\begin{equation}	\label{eq:optimal-trans2}
	P^*(p,q)_{ij} = u_i K_{ij} v_j. 
\end{equation}
By the Sinkhorn algorithm, we obtain the sequence $\Pt\in\Pnm$
defined by
\[
	\Pt_{ij} := \ut_i K_{ij} \vt_j,  
\]
where $\ut, \vt$ are vectors obtained in Algorithm \ref{alg1}. 
It is known that $\{\Pt\}$ converges to $P^*(p,q)$ as $t$ tends to infinity \cite[Section 4]{CuturiPeyre19}. 
This algorithm requires $O(\max\{n,m\}^2)$ time 
for each iteration of the \textbf{while} loop, 
while the linear programming costs $O(\max\{n,m\}^3)$ time. 

\begin{algorithm}[htbp]
	\caption{Sinkhorn algorithm}		\label{alg1}                          
	\begin{algorithmic}
		\STATE $p\in\calP[n-1], q\in\calP[m-1], K=(K_{ij})\in\R^{n\times m}, \lambda>0$: given
		\STATE $u^{(0)}\in\R^n, v^{(0)}\in\R^m$: initial values
		\WHILE{until converge}
			\FOR{$i=1$ to $n$}
				\STATE $\utt_i \Leftarrow p_i \big/ \( \sumn[m]{j} K_{ij} \vt_j \)$
			\ENDFOR
			\FOR{$j=1$ to $m$}
				\STATE $\vtt_j \Leftarrow q_j \big/ \( \sumn[n]{i} \utt_i K_{ij} \)$
			\ENDFOR
			\STATE $t \Leftarrow t+1$
		\ENDWHILE
	\end{algorithmic}
\end{algorithm}

\subsection{Information geometric perspective of entropic regularization} 
Amari \etal \cite{AmariKO} focused on the set of optimal transport plans
\[
	\calP^{opt} := \Set{P^*(p,q)\in\Pnm | p\in\calP[n-1], q\in\calP[m-1] }, 
\]
which is an exponential family with canonical parameters $(\alpha,\beta)$ 
as seen in Eq. \eqref{eq:optimal-trans}. 
In terms of the information geometry \cite{AmariNagaoka}, 
on the probability simplex $\Pnm$ with the dually flat structure $(\fish,\connm,\conne)$, 
$\calP^{opt}$ is a $\conne$-autoparallel submanifold, 
where $\fish$ is the Fisher metric, and $\connm, \conne$ are $m$-, $e$- connections, respectively.  

On the other hand, for each $p\in\calP[n-1], q\in\calP[m-1]$, 
submanifolds $M_{p,\cdot}, M_{\cdot,q}$ defined by
\begin{align*}
	M_{p,\cdot} := \Set{ P\in\Pnm | \sumn[m]{j}P_{ij} = p_i }, \quad
	M_{\cdot,q} := \Set{ P\in\Pnm | \sumn[n]{i}P_{ij} = q_j }
\end{align*}
are $\connm$-autoparallel. 
Moreover, they are orthogonal to $\calP^{opt}$ with respect to $\fish$. 
This fact allows us interpret the Sinkhorn algorithm as an iterative geometric operations. 
More precisely, it is stated as the following. 

\begin{proposition}[Amari \etal\cite{AmariKO}]
	Let $\{\ut\}_t, \{\vt\}_t$ be a sequence given by Algorithm \ref{alg1}, 
	and $\Pt, \Qt\in\Pnm$ defined as
	\[
		\Pt_{ij} := \ut_i K_{ij} \vt_j, \qquad
		\Qt_{ij} := \utt_i K_{ij} \vt_j, 
	\]
	for each $t\in\N$.
	Then, for each instant $t$, 
	$\Qt$ attains the $e$-projection of $\Pt$ onto $M_{p,\cdot}$, 
	and $\Ptt$ attains that of $\Qt$ onto $M_{\cdot,q}$. 
\end{proposition}


The Kullback-Leibler divergence, the canonical divergence on $(\Pnm,\fish,\connm,\conne)$, 
is given by
\[
	\KL{P}{Q} := \sum_{i,j} P_{ij} \log\frac{P_{ij}}{Q_{ij}}, \quad P,Q\in\Pnm.  
\]
Due to the projection theorem \cite[Theorem 3.9]{AmariNagaoka}, 
the $e$-projection $\Qt$ of $\Pt$ is given as
\begin{equation}	\label{eq:proj-Sinkhorn}
	\Qt = \underset{R\in M_{p,\cdot}}{\arg\min} \, \KL{R}{\Pt}.  
\end{equation}
In other words, $\Qt$ is the \textit{nearest} point to $\Pt$ lying in $M_{p,\cdot}$ 
in sense of the Kullback-Leibler divergence. 
In addition, 
from the Pythagorean theorem \cite[Theorem 3.8]{AmariNagaoka}, 
it holds that
\begin{align*}
	\KL{P^*(p,q)}{\Pt} 
	&= \KL{P^*(p,q)}{\Qt} + \KL{P^*(p,q)}{\Qt} \\
	&\geq \KL{P^*(p,q)}{\Qt}, 
\end{align*}
and the equality holds if and only if $\Pt=\Qt$. 
Combined with a similar argument for $\Ptt$, 
we obtain 
\[
	\KL{P^*(p,q)}{\Pt} \geq \KL{P^*(p,q)}{\Qt} \geq \KL{P^*(p,q)}{\Ptt},  
\]
and thus, the Kullback-Leibler divergence between $P^*(p,q)$ and $\Pt$ 
decreases strictly monotonically during the Sinkhorn algorithm (Figure \ref{fig4}). 

\begin{figure}[htbp]
\begin{center}
    	\includegraphics[width=0.5\textwidth]{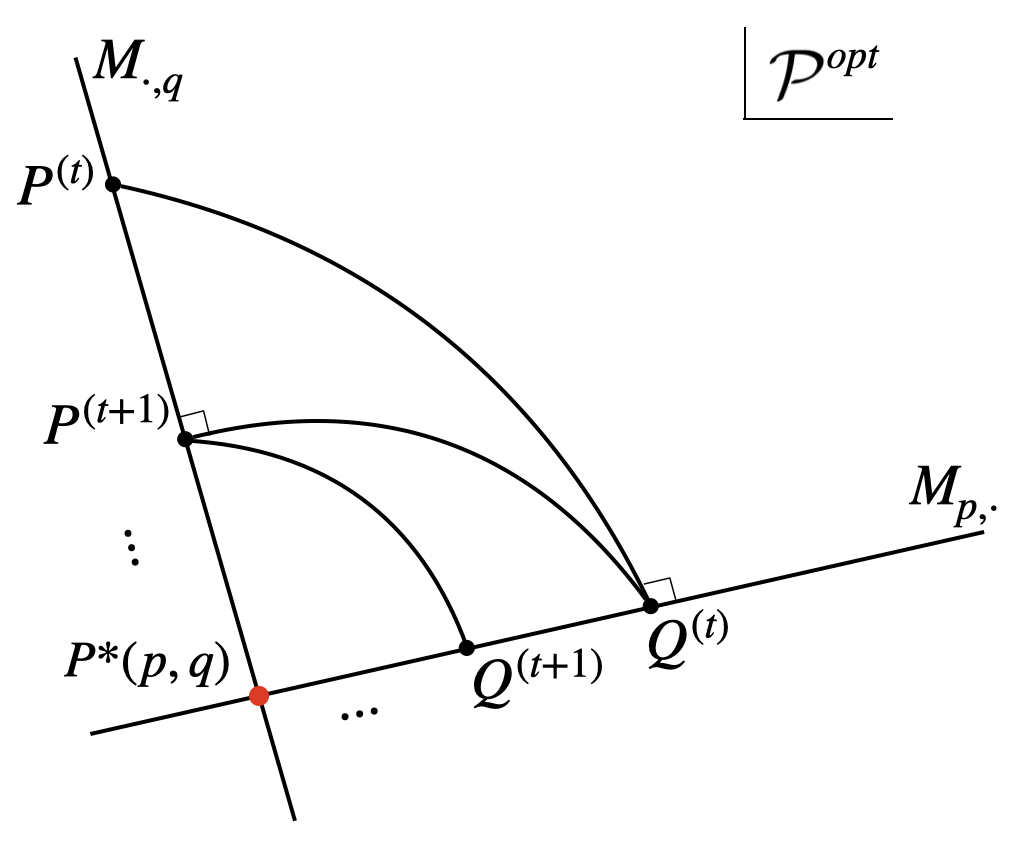}
	\caption{
		A schematic diagram for the geometric interpretation of the Sinkhorn algorithm. 
		The point $\Qt$ is the $e$-projection of $\Pt$ onto $M_{p,\cdot}$, 
		and $\Ptt$ is that of $\Qt$ onto $M_{\cdot,q}$. 
		All of $\Pt$ and $\Qt$ belong to $\calP^{opt}$ except for $P^{(0)}$.   
		They converge to the optimal plan $P^*(p,q)$, 
		which is located at the intersection of 
		$\calP^{opt}$ and $\Pi(p,q)=M_{p,\cdot}\cap M_{\cdot,q}$ (presented by the red point). 
	}
	\label{fig4} 
\end{center}
\end{figure}

\subsection{Barycenter problem} 
The minimized quantity $W_\lam(p,q)$ in the problem \eqref{eq:Cut-prob} 
is a strictly convex function on $(p,q)\in\calP[n-1]\times\calP[m-1]$
with respect to the standard affine structure. 
In fact, letting $p_t := (1-t)p_0+t p_1, q_t := (1-t)q_0+t q_1$, we have
\begin{align*}
	W_\lam(p_t, q_t) &= \inf_{P\in\Pi(p_t, q_t)}\Phi_\lam(P) \\
	&\leq \Phi_\lam\((1-t)P^*_0 + t P^*_1\) \\
	&\leq (1-t)\Phi_\lam(P^*_0) + t \Phi_\lam(P^*_1) \\
	&= (1-t)W_\lam(p_0, q_0) + t W_\lam(p_1, q_1),  
\end{align*}
where $P^*_0, P^*_1$ are optimal plans for $(p_0,q_0), (p_1,q_1)$ respectively. 
We used the convexity of $\Phi_\lam$ at the second inequality. 
The strict convexity of $W_\lam$ follows from that of $\Phi_\lam$. 

As $\lam$ tends to 0, $W_\lam(p,q)$ goes to 
the minimal cost $W(p,q)$ defined by \eqref{eq:Was-prob}. 
Assuming $m=n$ and some conditions for the cost matrix $C$, 
the quantity $W(p,q)$ can be regarded as a metric on $\calP_{n-1}$, 
called the \textit{Wasserstein distance} \cite{Villani03}. 
Then, one can consider the problem to compute the \textit{\Frechet~mean}
\begin{equation}	\label{eq:orig-bary}
	p^* := \arg\inf_{q\in\calP[n-1]} \sumn[N]{k} r_k \, W(p^k,q) , 
\end{equation}
of given points $p^1,\dots,p^N \in \calP[n-1]$ and weights $r_1,\dots,r_N\in\R_{++}$ 
with $\sumn[N]{k}r_k=1$.
Such a type of mean is called the \textit{Wasserstein barycenter}. 

Let us consider a relaxed variant of the problem \eqref{eq:orig-bary}:  
given $p^1,\dots,p^N \in \calP[n-1]$ and $r_1,\dots,r_N\in\R_{++}$ with $\sumn[N]{k}r_k=1$, 
\begin{equation}	\label{eq:bary}
	\textbf{Minimize } \sumn[N]{k} r_k \, W_\lam(p^k,q) \text{ under } q\in\calP[n-1] .
\end{equation}
Benamou \etal showed that this problem can be solved by a Sinkhorn-like algorithm \cite{BenamouCCNP}, 
which is presented in Algorithm \ref{alg2}. 
In order to devise the algorithm, 
they enlarged the domain of the problem to $(\Pnn)^N$ and considered the problem 
\[
	\textbf{Minimize } \sumn[N]{k} r_k \, \Phi_\lam(P^k) \text{ under } (P^1,\dots,P^N)\in \dM_1\cap \dM_2, 
\]
where 
\begin{align*}
	\dM_1 &= \Set{ (P^1,\dots,P^N)\in(\Pnn)^N | P^k\in M_{p^k,\cdot}, \any k }, \\
	\dM_2 &= \Set{ (P^1,\dots,P^N)\in(\Pnn)^N | \exs q\in\calP_{n-1} \st P^k\in M_{\cdot,q}, \any k }. 
\end{align*}
This problem is, in fact, equivalent to the original one \eqref{eq:bary}. 
Algorithm \ref{alg2} is obtained by computing the $\conne$-projection onto 
the pair of $\connm$-autoparallel submanifolds $\dM_1$ and $\dM_2$ of $(\Pnn)^N$, iteratively. 
Each iteration of the \textbf{while} loop costs $O(Nn^2)$
time. 

\begin{algorithm}[htbp]
	\caption{Benamou \textit{et al}.'s algorithm}		\label{alg2}                          
	\begin{algorithmic}
		\STATE $p^1,\dots,p^N\in\calP[n-1]$, $r_1,\dots,r_N\in\R_+$ with $\sum_k^N r_k = 1, 
			K\in\R^{n\times m}$: given
		\STATE $(u^{(1;0)},\dots,u^{(N;0)}),(v^{(1;0)},\dots,v^{(N;0)})\in\R^{n\times N}$: initial values
		\WHILE{until converge}
			\FOR{$k=1$ to $N$}
        			\FOR{$i=1$ to $n$}
        				\STATE $u^{(k;t+1)}_i \Leftarrow p^k_i \big/ \( \sumn{j} K_{ij} v^{(k;t)}_j \)$
        			\ENDFOR
			\ENDFOR
			\FOR{$j=1$ to $n$}
				\STATE $\tilde{p}_j \Leftarrow \prod_{k=1}^N \(\sumn{i}u^{(k;t+1)}_iK_{ij}\)^{r_k}$
			\ENDFOR
			\FOR{$k=1$ to $N$}
        			\FOR{$j=1$ to $n$}
					\STATE $v^{(k;t+1)}_j \Leftarrow \tilde{p}_j \big/  \( \sumn{i} u^{(k;t+1)}_i K_{ij} \)$
        			\ENDFOR
			\ENDFOR
			\STATE $t\Leftarrow t+1$ 
		\ENDWHILE
	\end{algorithmic}
\end{algorithm}


	\section{Geometric perspective of the problems and generalization of algorithms}

In this section, for given $p\in\calP_{n-1}$ and $q\in\calP_{m-1}$, we consider the minimization problem
\begin{equation}	\label{problem}
	\textbf{Minimize } \Phi(P) \text{ under } P \in \Pi(p,q),  
\end{equation}
where $\Phi$ is a strictly convex smooth function on $\Pnm$. 
We herein assume that $\Phi$ can be extended to a strictly convex smooth function $\tPhi$ 
on an open neighborhood $U$ of $\Pnm$ in $\Rnm$. 
This problem includes Cuturi's regularized optimal transport problem 
as the case where $\Phi=\Phi_\lam$. 
We investigate the dually flat structure suitable for the problem \eqref{problem}, 
and devise a procedure which generalize the Sinkhorn algorithm. 
We also address the generalized barycenter problem \eqref{eq:gen-bary}.

\subsection{Geometric perspective of the minimization problem}
On $\Pnm$, 
the Bregman divergence $D$ associated to $\Phi$ is given by
\begin{align*}
	D(P||Q) 
	:= \Phi(\eta(P)) - \Phi(\eta(Q)) - \<\pd{\Phi}{\eta}(\eta(Q)), \eta(P)-\eta(Q)\>, 
\end{align*}
where $\eta$ is an affine coordinate system on $\Pnm$ 
compatible with the standard affine structure. 
The partial derivative $\p\Phi/\p\eta$ defines the dual affine coordinate system $\theta$, i.e., 
\begin{equation}	\label{eq:dual-coord}
	\theta(P) := \pd{\Phi}{\eta}(\eta(P)), \quad \any P\in\Pnm. 
\end{equation}
The divergence $D$ induces a dually flat structure $(g,\nabla,\nabla^*)$ in the standard manner: 
\begin{align*}
	g_P(X,Y) &= -X_P Y_{Q} D(P||Q) \big|_{Q=P}, \\
	g_P(\nabla_XY,Z) &= -X_P Y_P Z_{Q} D(P||Q) \big|_{Q=P}, \\
	g_P(\nabla^*_XY,Z) &= -X_{Q} Y_{Q} Z_P D(P||Q) \big|_{Q=P}, 
\end{align*}
for $X,Y,Z\in \mathscr{X}(\Pnm)$ and $P\in\Pnm$, 
where $\mathscr{X}(\Pnm)$ denotes the set of vector fields on $\Pnm$. 
Here, the notation $X_P$ means that the vector field $X$ acts to the function $P\mapsto D(P||Q)$. 
The connections $\nabla$ and $\nabla^*$ are flat, 
and $\eta$ and $\theta$ are their affine coordinate systems, respectively. 
Moreover, $\eta$ and $\theta$ are mutually dual with respect to $g$, that is, 
they satisfy the relation
\begin{equation}	\label{eq:duality-of-coord}
	g\(\pd{}{\eta_i}, \pd{}{\theta^j}\) = \delta^i_j, 
\end{equation}
where $\delta^i_j$ denotes the Kronecker delta. 

In order to give a geometric presentation of the problem \eqref{problem}, 
we use the affine coordinate system $\eta$ defined by
\begin{align}	\label{eq:eta1}
\begin{cases}	
	&\etij(P) = P_{ij}, \quad \ran{i}{n-1}, \ran{j}{m-1}, \\
	&\etim(P) = \sumn[m]{j}P_{ij}, \quad \ran{i}{n-1}, \\
	&\etnj(P) = \sumn{i}P_{ij}, \quad \ran{j}{m-1}. 
\end{cases}
\end{align}
Under this coordinate system, the subset $\Pi(p,q)$ is presented as
\[
	\Pi(p,q) = M_{p,\cdot} \cap M_{\cdot, q}
\]
by using $\nabla$-autoparallel submanifolds defined by
\begin{align*}
	M_{p,\cdot} &:= \Set{ P\in\Pnm | \etim(P)=p_i, ~\ran{i}{n-1} }, \\
	M_{\cdot, q} &:= \Set{ P\in\Pnm | \etnj(P)=q_j, ~\ran{j}{m-1} }. 
\end{align*}
In particular, $\Pi(p,q)$ itself is also an $\nabla$-autoparallel submanifold. 
Then, the problem \eqref{problem}, if its solution lies in $\Pnm$, is reduced to
\begin{align}		\label{eq:critical}
\left\{ \begin{array}{l}
	\displaystyle \pd{\Phi}{\eta_{ij}}(\eta) = 0, \quad \ran{i}{n-1}, \ran{j}{m-1} \\[3mm]
	\displaystyle \eta_{im}=p_i, ~ \eta_{nj}=q_j, \quad \ran{i}{n}, \ran{j}{m} 
\end{array} \right. . 
\end{align}

From the relation \eqref{eq:dual-coord} and the inverse relation
\begin{align*}	
	\begin{cases}	
		&P_{ij} = \etij, \qquad \ran{i}{n-1}, \ran{j}{m-1}, \\
		&P_{im} = \etim - \sum_{j=1}^{m-1}\etij, \qquad \ran{i}{n-1}, \\
		&P_{nj} = \etnj - \sum_{i=1}^{n-1}\etij, \qquad \ran{j}{m-1}, \\
		&P_{nm} = 1 - \sum_{i=1}^{n-1}\etim - \sum_{j=1}^{m-1}\etnj 
			+ \sum_{i=1}^{n-1}\sum_{j=1}^{m-1}\etij,  
	\end{cases} 
\end{align*}
one can check that the dual affine coordinate system $\theta$ is given by 
\begin{align}	\label{eq:theta1}
\begin{cases}
	&\thij = S^{ij}(P) - S^{im}(P) - S^{nj}(P) + S^{nm}(P), \\
	&\hspace{25mm} \ran{i}{n-1}, \ran{j}{m-1} \\
	&\thim = S^{im}(P) - S^{nm}(P), \quad \ran{i}{n-1}, \\
	&\thnj = S^{nj}(P) - S^{nm}(P), \quad \ran{j}{m-1}.  
\end{cases}
\end{align}
Here, we denote by $S^{ij}:=\partial \tPhi/\partial A_{ij}$ for $1\leq i\leq n, 1\leq j\leq m$, 
where $\tPhi$ is an extension of $\Phi$ to $U\inc\Rnm$ and $A=(A_{ij})$ indicates an element of $\Rnm$. 
Thus, the critical condition 
\[
	\theta^{ij} = \pd{\Phi}{\eta_{ij}} = 0
\]
in \eqref{eq:critical} is rewritten as 
\[
	S^{ij}(P) - S^{im}(P) = S^{nj}(P) - S^{nm}(P), 
\]
which implies that $S^{ij}(P)-S^{im}(P)$ does not depend on $i$. 
By rearranging some terms, one can also check that $S^{ij}(P)-S^{nj}(P)$ is independent of $j$. 
As a consequence, there exist $\alpha\in\R^n$ and $\beta\in\R^m$ satisfying 
\begin{equation}	
	S^{ij}(P) = \alpha^i + \beta^j, \quad \ran{i}{n}, \ran{j}{m}.  
\end{equation}
From \eqref{eq:theta1}, we obtain 
\begin{align*}
	\theta^{im}(P) &= \alpha^i - \alpha^n, \quad \ran{i}{n-1}, \\
	\theta^{nj}(P) &= \beta^j - \beta^m, \quad \ran{j}{m-1},   
\end{align*}
and thus, the quantities $\theta^{im}$ and $\theta^{nj}$ are equivalent variables 
to $\alpha^i$ and $\beta^j$, up to additive constants $\alpha^n, \beta^m$. 

Let us consider a $\nabla^*$-autoparallel submanifold defined by
\[
	\calP^{opt} := \Set{\theta^{ij}=0, \quad 1\leq i\leq n-1, 1\leq j\leq m-1}. 
\]
Each element $P\in\calP^{opt}$ lies in $\Pi(p,q)$ for some $p\in\calP_{n-1},q\in\calP_{m-1}$, 
and then, $P$ is the solution of the problem \eqref{problem} for those $p,q$. 
Therefore, $\calP^{opt}$ is the set of optimal solutions for some source and target distributions. 
From the duality \eqref{eq:duality-of-coord} of the coordinate systems $\eta$ and $\theta$, 
$\calP^{opt}$ is orthogonal to $\Pi(p,q)$ with respect to $g$. 
Now, the problem \eqref{problem} is interpreted as 
the problem of finding the intersection point 
between the $\nabla^*$-autoparallel submanifold $\calP^{opt}$
and the $\nabla$-autoparallel submanifold $\Pi(p,q)$, which are mutually orthogonal. 

To describe such a type of problem, 
the concept of a mixed coordinate system is useful. 
\begin{proposition}[Mixed coordinate system]		\label{thm:mixed}
	Let $(M,g,\nabla,\nabla^*)$ be an $n$-dimensional dually flat manifold, 
	and $\eta=(\eta_i)$ and $\theta=(\theta^i)$ be 
	affine coordinate systems of $\nabla$ and $\nabla^*$, respectively. 
	Suppose that $\eta$ and $\theta$ are mutually dual. 
	Then, for $1\leq k\leq n$, 
	\[
		\xi = (\theta^1,\dots,\theta^k, \eta_{k+1},\dots,\eta_n)
	\]
	becomes a coordinate system on $M$, which is called a mixed coordinate system. 
\end{proposition}

We use the mixed coordinate system $\xi=(\xi^{ij})$ defined by
\begin{align*}
\begin{cases}
	\xi^{ij}(P) = \theta^{ij}(P), \quad \ran{i}{n-1},\ran{j}{m-1}, \\
	\xi^{im}(P) = \eta^{im}(P), \quad \ran{i}{n-1}, \\
	\xi^{nj}(P) = \eta^{nj}(P), \quad \ran{j}{m-1}. 
\end{cases}
\end{align*} 
Under this coordinate system, the equation \eqref{problem} is presented by
\begin{align*}
	\xi^{ij}(P) = 0, \quad 
	\xi^{im}(P) = p_i, \quad 
	\xi^{nj}(P) = q_j, \quad
	\ran{i}{n-1},\ran{j}{m-1}. 
\end{align*} 
From the above discussion, the solution $P^*(p,q)$ of \eqref{problem} is,
if it exists, given by
\[
	S^{ij}(P^*(p,q)) = (\alpha^*)^i + (\beta^*)^j,  
\]
and $\alpha^*\in\R^n$ and $\beta^*\in\R^m$ are determined by the conditions
\[
	\eta_{im}(P^*(p,q))=p_i, \quad 
	\eta_{nj}(P^*(p,q))=q_j, \quad \ran{i}{n}, \ran{j}{m},  
\]
due to Proposition \ref{thm:mixed}.

\subsection{A Generalization of the Sinkhorn algorithm}
We herein assume that the convex function $\Phi$ has a smooth extension $\tPhi$ onto $U=\Rnm$, 
and that $S:\Rnm\to\R^{n\times m}$ is surjective. 
For example, the function $\Phi_\lam$ defined by \eqref{eq:phi-lambda} has the extension
\begin{align}
	\label{eq:extend-shannon}
	\tPhi_\lam(A) &:= \<A,C\> - \lam \tH(A), \\
	\tH(A) &:= -\sum_{i,j}A_{ij}\log A_{ij} + \(\sum_{i,j} A_{ij}-1\), \quad A\in\Rnm, \notag
\end{align}
which satisfies the assumption. 
These assumption are slightly too strong for introducing an information geometric structure; 
however, those assumptions are necessary for a straightforward generalization of the entropic regularization, 
including algorithms which solve it numerically. 

Before we discuss the generalization of the Sinkhorn algorithm, 
let us introduce a dual problem of the problem \eqref{problem}. 
The next lemma is a relaxed version of the Kantorovich duality \cite{Villani03}, 
which is a well-known theorem in the optimal transport theory. 
A proof of the lemma is located in Appendix. 

\begin{lemma}	\label{thm:dual}
	Let $\tPhi : \R^{n\times m}_{++} \to \R$ be a convex function. 
	For $p\in\calP[n-1]$, $q\in\calP[m-1]$, 
	\[
		\inf_{P\in\Pi(p,q)} \tPhi(P) = 
			\sup_{\alpha\in\R^n, \beta\in\R^m} \{\<p,\alpha\>+\<q,\beta\> -\tPhi^*(\alpha\oplus\beta)\} , 
	\]
	where $\tPhi^*:\R^{n\times m} \to \R\cup\{+\infty\}$ is the Legendre transform of $\tPhi$
	defined by 
	\[
		\tPhi^*(u):= \sup_{A\in\Rnm}\{ \<A,u\>-\tPhi(A) \}, \quad u\in\R^{n\times m}, 
	\]
	and $\alpha\oplus\beta \in \R^{n\times m}$ is defined by 
	$(\alpha\oplus\beta)^{ij} := \alpha^i + \beta^j$. 
\end{lemma}

In our setting, since $\tPhi$ is an extension of $\Phi$, 
the left hand side in {Lemma \ref{thm:dual}} is equal to that of \eqref{problem}. 
Let us note that the dual problem always has a solution, 
which is guaranteed by Lemma \ref{thm:dual-sol-exist} in Appendix. 
We also note that the assumption that the domain of $\tPhi$ is $\Rnm$ is essential for obtaining this lemma. 
Using the above lemma, we obtain the following theorem, 
which guarantees that the primal solution $P^*(p,q)$ is located on the interior of the domain, 
namely, 
\[
	\inf_{P\in\Pi(p,q)} \Phi(P) = \min_{P\in\Pi(p,q)} \Phi(P). 
\]

\begin{theorem}	\label{thm:onto}
	Suppose that $\tPhi:\Rnm\to\R$ is smooth and strictly convex 
	and that $S:\Rnm\to\R^{n\times m}$ is surjective. 
	There exists a unique solution $P^*(p,q) \in \Pi(p,q)$ of the minimization problem \eqref{problem} 
	 for each $p\in\calP[n-1], q\in\calP[m-1]$. 
	Moreover, a pair of dual solutions $(\alpha^*,\beta^*)\in\R^n\times\R^m$ satisfies
	\[
		S^{ij}(P^*(p,q)) = (\alpha^*)^i + (\beta^*)^j .
	\]
\end{theorem}
\begin{proof}
	Since $\Phi$ is strictly convex and the closure $\overline{\Pi(p,q)}$ is compact, 
	\[
		P^*(p,q) := \underset{P\in\Pi(p,q)}{\arg\inf} \Phi(P) 
	\]
	exists uniquely in $\overline{\Pi(p,q)}$. 
	From the surjectivity of $S$, we can choose $A_*\in S^{-1}(\alpha^*\oplus\beta^*)$. 
	We show that $P^*(p,q)=A_*$. 
	
	Due to Lemma \ref{thm:dual}, 
	the primal and dual solutions $P^*(p,q)$, $(\alpha^*,\beta^*)$ attain the equality 
	\[
		\hPhi(P^*(p,q)) = \<P^*,\alpha^*\oplus\beta^*\> - \tPhi^*(\alpha^*\oplus\beta^*),   
	\]
	where $\hPhi(P^*(p,q)):=\lim_{\tA\to P^*(p,q)}\tPhi(\tA)$.
	This implies that $\alpha^*\oplus\beta^*\in\R^{n\times m}$ is a subgradient of $\tPhi$ at $P^*(p,q)$. 
	On the other hand, by the choice of $A_*$, 
	$\alpha^*\oplus\beta^*$ is also a subgradient of $\tPhi$ at $A_*$. 
	Hence, for any $t\in(0,1)$, we have 
	\begin{align*}
		&\tPhi(tP^*(p,q)+(1-t)A_*) \\
		&\geq \< (tP^*(p,q)+(1-t)A_*) - P^*(p,q), \alpha^*\oplus\beta^* \> + \hPhi(P^*(p,q)) \\
		&= \< tP^*(p,q)+(1-t)A_*, \alpha^*\oplus\beta^* \> - \tPhi^*(\alpha^*\oplus\beta^*) \\
		&= t\( \<P^*(p,q),\alpha^*\oplus\beta^* \> - \tPhi^*(\alpha^*\oplus\beta^*) \) \\
			&\hs +(1-t)\(\<A_*, \alpha^*\oplus\beta^* \> - \tPhi^*(\alpha^*\oplus\beta^*) \) \\
		&= t\hPhi(P^*(p,q))+(1-t)\tPhi(A_*). 
	\end{align*}
	From the strict convexity of $\tPhi$, 
	if $P^*(p,q)\neq A_*$, it holds that
	\[
		\tPhi(tP^*(p,q)+(1-t)A_*) < t\hPhi(P^*(p,q))+(1-t)\tPhi(A_*),  
	\]
	which leads to a contradiction. 
	Hence, we obtain the former assertion
	\[
		P^*(p,q) = A_* \in \overline{\Pi(p,q)}\cap\Rnm = \Pi(p,q). 
	\]
	
	The latter one follows from the smoothness of $\tPhi$, 
	since $\alpha^*\oplus\beta^*$ is a subgradient of $\tPhi$ at $P^*(p,q)$. 
\end{proof}

The Sinkhorn algorithm is generalized 
as the iterative procedure consisting of 
\begin{enumerate}[(S-I)]
	\item the $\nabla^*$-projection 
		onto the $\nabla$-autoparallel submanifold $M_{p,\cdot}$ and
	\item the $\nabla^*$-projection 
		onto the $\nabla$-autoparallel submanifold $M_{\cdot,q}$. 
\end{enumerate}
Due to the pythagorean theorem, the Bregman divergence 
from the solution $P^*(p,q)$ will decrease monotonically for each iteration.  

In terms of dual affine coordinate systems, this procedure is written as follows. 
Suppose that $\Pt$ is given. 
Then, using the mixed coordinate system $(\eta_{im}, \theta^{ij}, \theta^{nj})$, 
the first projection $\Qt$ is designated by the coordinate
\begin{align}	\label{eq:projection}	
\begin{aligned}
	\etim(\Qt)&=p_i, \quad 1\leq i \leq n-1, \\
	\thij(\Qt)&=0, \quad \thnj(\Qt)=\thnj(\Pt), \quad 1\leq i\leq n-1, 1\leq j\leq m-1,   
\end{aligned}
\end{align}
For $\alpha\in\R^n, \beta\in\R^m$, let $A(\alpha,\beta)$ denote an element of $\Rnm$ satisfying 
\begin{equation}	\label{eq:alpha-beta}
	S^{ij}(A(\alpha,\beta)) = \alpha^i + \beta^j, \quad 1\leq i\leq n, 1\leq j\leq m. 
\end{equation}
Then, letting $\Pt=A(\at,\bt)$, the projection is given by 
\[
	\Qt = A(\att,\bt),
\] 
where $\att$ is determined by the equation 
\begin{align}	\tag*{(S-I)$'$}
	\etim(A(\att,\bt))&=p_i, \quad 1\leq i \leq n-1. 
\end{align}
Let us remark that the second equation $\thij=0$ in \eqref{eq:projection} is automatically satisfied
because of the equation \eqref{eq:alpha-beta}, 
and that the third equation $\thnj(\Qt)=\thnj(\Pt)$ means that $\bt$ is fixed. 
Similarly, the second projection $\Ptt=A(\att,\btt)$ is obtained by solving the equation 
\begin{align}	\tag*{(S-II)$'$}
	\etnj(A(\att,\btt))&=q_j, \quad 1\leq j\leq m-1. 
\end{align}

\begin{example}[Sinkhorn algorithm (Algorithm \ref{alg1})]
	The partial derivative of the function defined by \eqref{eq:extend-shannon} is
	\[
		S^{ij}(A) = C^{ij} + \lam \log A_{ij}, \quad A\in\Rnm. 
	\]
	Due to Theorem \ref{thm:onto}, we obtain that
	\begin{align*}
		P^*(p,q)_{ij} &= \exp\(\1{\lam}( \alpha_i + \beta_j - C^{ij})\) = A(\alpha,\beta) \\
			&= u_i K_{ij} v_j , 
	\end{align*}
	for some $\alpha\in\R^n, \beta\in\R^m$, 
	where we put $u_i=\exp(\alpha_i/\lam), v_j=\exp(\beta_j/\lam)$, 
	and $K_{ij}:=\exp(-C^{ij}/\lam)$. 
	Then, given $\Pt=(\ut )^T K\vt$, the equation (S-I)$'$ is reduced to
	\begin{align*}
		\utt_i \bigg(\sum_j K_{ij} \vt_j\bigg) =p_i, \quad 1\leq i \leq n, 
	\end{align*}
	and it is solved to
	\[
		\utt_i = p_i / (K \vt)_i.  
	\]
	This shows that our algorithm actually give a generalization of the Sinkhorn algorithm. 
\end{example}

\subsection{Geometric perspective for the generalized barycenter problem}
We assume the surjectivity of $S:\Rnm\to\R^{n\times m}$ also in this subsection. 
Then, a $\nabla$-affine coordinate system on the submanifold $\calP^{opt}$ 
is given by $(\etim,\etnj)=(p_i,q_j)$, since $\calP^{opt}$ is characterized by $\theta^{ij}=0$.  
Let $\vphi$ be a function on $\calP_{n-1}\times\calP_{m-1}$ defined by 
\begin{align*}
	\vphi(p,q) 
		&:= \inf_{P\in\Pi(p,q)} \Phi(P). 
\end{align*}
Then, $\vphi$ is strictly convex because of the strict convexity of $\Phi$.  
The function $\vphi$ is a restriction of $\Phi$ onto $\calP^{opt}$, 
and it is a potential function for the dually flat structure on $\calP^{opt}$. 
In fact, since $\thij(P^*(p,q))=0$, 
\begin{align}
	\pd{\vphi}{p_k}(p,q)
	&= \sumn[n-1]{i}\sumn[m-1]{j} \pd{\Phi}{\etij}(P^*(p,q))\pd{P_{ij}^*(p,q)}{p_k} \notag\\ 
		&\hs + \sumn[n-1]{i} \pd{\Phi}{\etim}(P^*(p,q))\pd{p_i}{p_k} 
		+ \sumn[m-1]{j} \pd{\Phi}{\etnj}(P^*(p,q))\pd{q_j}{p_k} \notag\\
	&= \sum_{i,j} \thij(P^*(p,q))\pd{P_{ij}^*(p,q)}{p_k} + \sum_i \thim(P^*(p,q))\delta_i^k \notag\\
	&= \theta^{km}(P^*(p,q)). 
	\label{eq:restrict-potential}
\end{align}
Similarly, another relation $\p\vphi/\p q_l = \theta^{nl}$ follows.  

We herein assume that $m=n$, 
and consider the barycenter problem \eqref{eq:gen-bary}. 
If the solution lies in the interior of the domain, 
this problem is interpreted as solving the critical condition
\[
	\pd{}{q_j} \( \sumn[N]{k} r_k \, \varphi(p^k,q) \) = 0, \quad 1 \leq j \leq n,
\]
due to the convexity of $\varphi$. 
Because of the relation \eqref{eq:restrict-potential}, 
we can further interpret the problem as the system of equations on $(P^1,\dots, P^N) \in (\Pnn)^N$: 
\begin{align}	\label{eq:critical-bary}
 	\left\{ \begin{array}{l}
		\displaystyle \eta_{in}(P^k) = p^k_i, \\[1.5ex]
		\displaystyle \eta_{nj}(P^1) = \cdots = \eta_{nj}(P^N) \ (= q_j), \\[1.0ex]
		\displaystyle \sumn[N]{k} r_k \, \theta^{nj}(P^k) = 0, \\[3.0ex]
		\displaystyle \theta^{ij}(P^k) = 0, 
	\end{array} \right.
	\qquad 1 \leq i,j \leq n-1, 1 \leq k \leq N. 
\end{align}

In order to illustrate the geometric view of the barycenter problem, 
we consider the dually flat structure $(\cg,\cdel,\cdel^*)$ on $(\Pnn)^N$ 
induced by the convex function
\[
	\cPhi(P^1,\dots,P^N) := \sumn[N]{k} r_k \Phi(P^k). 
\]
We choose a $\cdel$-affine coordinate system given by
\begin{align}	\label{eq:eta2}
\begin{cases}
	&\Etij(P^1,\dots,P^N) = \etij(P^k), \quad 1\leq k \leq N, \\
	&\Etin(P^1,\dots,P^N) = \etin(P^k), \quad 1\leq k \leq N, \\
	&\Etnj(P^1,\dots,P^N) = \etnj(P^k) - \etnj(P^{k+1}), \quad 1\leq k \leq N-1, \\
	&\Etnj[N](P^1,\dots,P^N) = \sumn[N]{k} \etnj(P^k), 
\end{cases}
\end{align}
for $1\leq i,j \leq n-1$, where $\etij, \etim, \etnj$ are defined in \eqref{eq:eta1}. 
Then, the dual affine coordinate $\Theta$ is, 
from the general relation $\Theta = \p\cPhi/\p H$, 
given by
\begin{align}	\label{eq:theta2}
\begin{cases}
	&\Thij = r_k \thij(P^k), \quad 1 \leq k \leq N,\\
	&\Thin = r_k \thin(P^k), \quad 1 \leq k \leq N, \\
	&\Thnj(P) = r_k \thnj(P^k) - r_{k+1} \thnj(P^{k+1}), \quad 1\leq k \leq N-1, \\
	&\Thnj[N](P) = \sumn[N]{k} r_k \thnj(P^k).  
\end{cases}
\end{align}
We treat another minimization problem 
\begin{equation}	\label{problem3}
	\textbf{Minimize } \cPhi(P^1,\dots,P^N) \text{ under } (P^1,\dots,P^N)\in\dM_1\cap\dM_2, 
\end{equation}
and show that this problem is equivalent to \eqref{eq:gen-bary}. 
Here, $\dM_1$ and $\dM_2$ are $\cdel$-autoparallel submanifolds of $(\Pnn)^N$ defined by
\begin{align*}
	\dM_1 &:= \Set{(P^1,\dots,P^N)\in(\Pnn)^N | \Etin(P^1,\dots,P^N)=p^k_i, 
		\begin{array}{l} 1\leq i\leq n-1,\\ 1\leq k\leq N \end{array} }, \\
	\dM_2 &:= \Set{(P^1,\dots,P^N)\in(\Pnn)^N | \Etnj(P^1,\dots,P^N)=0, 
		\begin{array}{l} 1\leq j\leq n-1,\\ 1\leq k\leq N-1 \end{array} }.  
\end{align*}
Making use of the dual coordinate system defined in \eqref{eq:eta2} and \eqref{eq:theta2},  
 the critical condition \eqref{eq:critical-bary} is interpreted as
\begin{align}	\label{eq:critical-bary2}
	\left\{ \begin{array}{l}
		\displaystyle \Etin(P^1,\dots,P^N) = p^k_i, \quad 1 \leq k \leq N, \\[2mm]
		\displaystyle \Etnj(P^1,\dots,P^N) = 0, \quad 1 \leq k \leq N-1, \\[2mm]
		\displaystyle \Thnj[N](P^1,\dots,P^N) = 0, \\[2mm]
		\displaystyle \Thij(P^1,\dots,P^N) = 0, \quad 1 \leq k \leq N, \quad 1 \leq i,j \leq n-1,  
	\end{array} \right. 
\end{align} 
which is no other than the critical condition  
\begin{align*}
	(P^1,\dots,P^N)\in\dM_1\cap\dM_2, \quad 
	\pd{\cPhi}{\Etnj[N]} = 0, \quad \pd{\cPhi}{\Etij} = 0
\end{align*} 
of the problem \eqref{problem3}. 
In this mean, the problem \eqref{problem3} is another form of \eqref{eq:gen-bary}. 
Hence, via the problem \eqref{problem3}, 
the barycenter problem \eqref{eq:gen-bary} is interpreted as 
the problem finding the intersection point 
between $\cdel$-autoparallel submanifold $\dM_1\cap\dM_2$
and $\cdel^*$-autoparallel submanifold 
\[
	\Set{\Thnj[N] = 0, \Thij = 0, 1\leq i,j\leq n-1, 1\leq k\leq N}.  
\]

\subsection{Generalized algorithm for computing barycenter}
Analogously to the Sinkhorn algorithm, 
the critical condition \eqref{eq:critical-bary2} 
can be solved by an iterative procedure, 
which is implemented as 
\begin{enumerate}[(B-I)]
	\item the $\cdel^*$-projection $\Qt$ of $\Pt$
		onto the $\cdel$-autoparallel submanifold $\dM_1$, and
	\item the $\cdel^*$-projection $\Ptt$ of $\Qt$
		onto the $\cdel$-autoparallel submanifold $\dM_2$. 
\end{enumerate}
In fact, due to the pythagorean theorem, 
the Bregman divergence on $(\Pnn)^N$ associated to $\cPhi$ 
monotonically decreases with the procedure. 

Suppose that $\Pt=(\PtN{1},\dots,\PtN{N})$ satisfies 
\[
	\Thij(\Pt) = 0, ~\Thnj[N](\Pt) = 0, \quad 1\leq i,j\leq n-1, 1\leq k\leq N. 
\] 
Then, the $\cdel^*$-projection $\Qt$ of $\Pt$ onto $\dM_1$ is presented 
by a mixed coordinate system of \eqref{eq:eta2} and \eqref{eq:theta2} as 
\begin{align}	\tag*{(B-I)$'$}
\begin{aligned}
	\Etin(\Qt)=p^k_i, \quad \Thij(\Qt)=0 (=\Thij(\Pt)), \quad \Thnj(\Qt)=\Thnj(\Pt), \\
	 1\leq i,j\leq n-1, 1\leq k\leq N. 
\end{aligned}
\end{align}
On the other hand, the second projection (B-II) onto $\dM_2$ is given by 
\begin{align} 	\tag*{(B-II)$'$}
\begin{aligned}
	&\Etnj(\Ptt)=0, \quad 1\leq k\leq N-1, \\
	&\Thij(\Ptt)=0 (=\Thij(\Qt)), \quad \Thin(\Ptt)=\Thin(\Qt), \quad 1\leq k\leq N, \\
	&\Thnj[N](\Ptt)=0 (=\Thnj[N](\Qt)), 
\end{aligned}
\end{align}
where $1\leq i,j\leq n-1$. 

In terms of $\eta$ and $\theta$ given in \eqref{eq:eta1} and \eqref{eq:theta1}, 
these results are represented as follows: 
\begin{itemize}
	\item finding $\Qt=(\QtN{1},\dots,\QtN{N})$ solving 
		\begin{equation}	\tag*{(B-I)$''$}
			\etin(\QtN{k})=p^k_i, \quad \thij(\QtN{k})=0, \quad \thnj(\QtN{k})=\thnj(\PtN{k}),  
		\end{equation}
		for $1\leq i,j \leq n-1, 1\leq k\leq N$, and 
	\item finding $\Ptt=(\PttN{1},\dots,\PttN{N})$ solving 
		\begin{align}	\tag*{(B-II)$''$}
		\begin{aligned}
			&\etnj(\PttN{1})=\cdots=\etnj(\PttN{N}), \quad \sum_{l=1}^N r_l \thnj(\PttN{l}) = 0, \\
			&\thij(\PttN{k})=0, \quad \thin(\PttN{k})=\thin(\QtN{k}),   
		\end{aligned}
		\end{align}
		for $1\leq i,j \leq n-1, 1\leq k\leq N$.  
\end{itemize}
By using the representation \eqref{eq:alpha-beta}, we can further reduce each procedure. 
Letting 
\[
	\PtN{k} = A\(\atN{k},\btN{k}\), 
\]
the algorithm is written as 
\begin{itemize}
	\item finding $(\attN{1},\dots,\attN{N})$ solving 
		\begin{equation}	\tag*{(B-I)$'''$}
			\etin\(A\(\attN{k},\btN{k}\)\)=p^k_i, 
		\end{equation}
		for $1\leq i \leq n-1, 1\leq k\leq N$, and 
	\item finding $(\bttN{1},\dots,\bttN{N})$ solving 
		\begin{align}	\tag*{(B-II)$'''$}
		\begin{aligned}
			&\etnj\(A\(\attN{1},\bttN{1}\)\)=\cdots=\etnj\(A\(\attN{N},\bttN{N}\)\), \\
			&\quad \sum_{k=1}^N r_k \(\bttN{k}\)^j = 0, 
		\end{aligned}
		\end{align}
		for $1\leq j \leq n-1, 1\leq k\leq N$. 
\end{itemize}

When the system of $N$ equations (B-II)$'''$ is hard to solve, 
one can avoid that difficulty by splitting the projection onto $\dM_2$. 
Let us consider the $\cdel^*$-projection onto 
\[
	\dM_{2;k} := \Set{(P^1,\dots,P^N)\in(\Pnn)^N | \Etnj(P^1,\dots,P^N)=0, 1\leq j\leq n-1}, 
\]
for $1\leq k\leq N-1$. 
Since
\[
	\dM_2 = \bigcap_{k=1}^{N-1} \dM_{2;k}, 
\]
due to the pythagorean theorem, 
a series of iterative $\cdel^*$-projections onto the $\cdel$-autoparallel submanifolds $\dM_{2;k}$
decreases monotonically the divergence from $\dM_2$. 
Hence, instead of computing (B-II)$'''$ directly, we can utilize alternative procedure
by solving
\begin{align}	\tag*{(B-II)$_k$}
\begin{aligned}
	&\etnj\(A\(\attN{k},\bttN{k}\)\)=\etnj\(A\(\attN{k+1},\bttN{k+1}\)\), \\
	& r_k \(\bttN{k}\)^j + r_{k+1} \(\bttN{k+1}\)^j = -\sum_{l\neq k, k+1}r_l \(\bttN{l}\)^j, 
\end{aligned}
\end{align}
with fixing $(\attN{1},\dots,\attN{N})$ and $(\bttN{1},\dots,\bttN{k-1},\bttN{k+2},\dots,\bttN{N})$.

\subsection{Another geometric perspective with 1-homogeneous extension}		\label{sec3-5}
Let us introduce another generalization of Amari-Cuturi's framework, 
using 1-homogeneous extension of the convex function $\Phi$ on $\Pnm$, 
which works well to solve the Tsallis entropic regularized optimal transport problem. 
Fix $\tq>0$ with $\tq\neq 1$ and $\lam>0$.   
We consider the problem
\[
	\textbf{Minimize } \Phi(P)=\<C,P\> - \lam \calT_{\tq}(P) \text{ under } P \in \Pi(p,q),  
\]
where $\calT_{\tq}$ denotes the $\tq$-Tsallis entropy, which is given by
\[
	\calT_{\tq}(P) = \1{\tq-1}\(1-\sum_{i,j} P_{ij}^{\tq}\),  
\]
This problem is originally considered in \cite{MuzellecNPN}. 

With the simple extension $A\mapsto \sum_{i,j} (\tq-1)^{-1}(1-A_{ij}^{\tq})$ of the Tsallis entropy, 
the range of 
\[
	S^{ij}(A) = \pd{\tPhi}{A_{ij}} = C^{ij} + \frac{\lam \tq}{\tq-1} A_{ij}^{\tq-1}
\]
becomes the subset 
\[
	\Set{u\in\R^{n\times m}|u_{ij}>C^{ij}},
\]  
which violates the assumption that $S$ is surjective. 
We herein consider the extension of $\calT_{\tq}$ defined by
\[
	\tT_{\tq}(A) = \1{\tq-1}\sum_{i,j}\( A_{ij} - \bigg(\sum_{k,l} A_{kl}\bigg)^{1-\tq} A_{ij}^{\tq}\), \quad A\in\Rnm. 
\]
which is 1-homogeneous, that is, $\tT_{\tq}(tP)=t\calT_{\tq}(P)$ for any $t>0$ and $P\in\Pnm$. 
Then, the associated mapping $S : \Rnm \to \R^{n\times m}$ is given by
\begin{equation}	\label{eq:tsallis-score}
	S^{ij}(t P) = S^{ij}(P) = C^{ij} + \frac{\lam}{\tq-1}\bigg( \tq P_{ij}^{\tq-1} + (1-\tq)\sum_{k,l} P_{kl}^{\tq} - 1 \bigg),  
\end{equation}
for $t>0, P\in\Pnm$. 

In general, a 1-homogeneous convex function induces a dually flat structure,  
which is called Dawid's decision geometry \cite{Dawid07}. 
For a 1-homogeneous $\tPhi$, its derivative $S:\Rnm \to \R^{n\times m}$ induces a mapping
from $\Pnm$ to $\R^{n\times m} / \<1_{nm}\>$. 
Here, $\R^{n\times m} / \<1_{nm}\>$ denotes a quotient vector space
divided by
\[
	u \sim v \iff u-v = c\, 1_{nm} \text{ for some $c\in\R$}, 
\]
where $1_{nm}$ denotes the matrix whose entries are all 1. 
Instead of Theorem \ref{thm:onto}, in this case, one can utilize the next theorem, 
whose proof is located in Appendix. 
\begin{theorem}	\label{thm:happy}
	Suppose that $\tPhi:\Rnm\to\R$ is 1-homogeneous and is strictly convex on $\Pnm$, 
	and that its derivative $S : \Rnm \to \R^{n\times m}$ induces a bijection between
	\[
		\Pnm \cong \R^{n\times m}/\<1_{nm}\>.  
	\]
	Then, there exists a unique solution $P^*(p,q) \in \Pi(p,q)$ of the minimization problem \eqref{eq:gen-prob} 
	 for each $p\in\calP[n-1], q\in\calP[m-1]$. 
	Moreover, there exists a pair $(\alpha^*,\beta^*)\in\R^n\times\R^m$ satisfying
	\[
		S(P^*(p,q))^{ij} = (\alpha^*)^i + (\beta^*)^j.
	\]
\end{theorem}

Applying Theorem \ref{thm:happy} to the mapping \eqref{eq:tsallis-score}, 
there are $\alpha\in\R^n, \beta\in\R^m$ satisfying
\begin{gather*}
	C^{ij} + \lam\bigg( \frac{\tq}{\tq-1} P^*(p,q)_{ij}^{\tq-1} - \kappa \bigg) = \alpha^i + \beta^j, \\
	\kappa := \sum_{k,l} P^*(p,q)_{kl}^{\tq} + \1{\tq-1}. 
\end{gather*}
We can include $\kappa$ in $(\alpha,\beta)$ 
by replacing $\alpha^i$ with $\alpha^i+\lam\kappa$. 
Thus, the optimal plan has the form as
\begin{align*}
	P^*(p,q)_{ij} = \( \frac{\tq-1}{\lam \tq}\(\alpha^i+\beta^j - C^{ij}\) \)^{\1{\tq-1}}.  
\end{align*}
The first projection (S-I) in the generalized Sinkhorn algorithm is interpreted in this case as solving 
\[
	\sumn[m]{j}\( \frac{\tq-1}{\lam \tq}\(\alpha^i+\beta^j - C^{ij}\) \)^{\1{\tq-1}} = p_i, \quad 1\leq i\leq n, 
\]
for the variable $\alpha$ with fixing $\beta$. 
In practice, this can be solved by the Newton method, for example. 
The second projection (S-II) is similarly given by solving
\[
	\sumn[n]{i}\( \frac{\tq-1}{\lam \tq}\(\alpha^i+\beta^j - C^{ij}\) \)^{\1{\tq-1}} = q_j, \quad 1\leq j\leq m, 
\]
with fixing $\alpha$. 

For the barycenter problem, the procedure (B-I) is similarly given as the equation
\[
	\sumn[n]{j}\( \frac{\tq-1}{\lam \tq}\((\alpha^k)^i+(\beta^k)^j - C^{ij}\) \)^{\1{\tq-1}} = p^k_i, \quad 
	1\leq i\leq n, 1\leq k\leq N, 
\] 
for $(\alpha^1,\dots,\alpha^N)$. 
The procedure (B-II) is given by
\begin{align*}
	\left\{ \begin{array}{l}
        	\displaystyle 
		\sumn[n]{i}\( \frac{\tq-1}{\lam \tq}\((\alpha^1)^i+(\beta^1)^j - C^{ij}\) \)^{\1{\tq-1}}
        	= \cdots
        	= \sumn[n]{i}\( \frac{\tq-1}{\lam \tq}\((\alpha^N)^i+(\beta^N)^j - C^{ij}\) \)^{\1{\tq-1}}, \\ 
        	\displaystyle \sumn[N]{k}r_k(\beta^k)^j = 0, \qquad 1\leq j\leq n. 
	\end{array} \right. 
\end{align*}
However, this equation is hard to solve. 
As an alternative way, we can make use of the procedure (B-II)$_k$, which is the equation
\begin{align*}
	\left\{ \begin{array}{l}
        	\displaystyle 
        	\sumn[n]{i}\( \frac{\tq-1}{\lam \tq}\((\alpha^k)^i+(\beta^k)^j - C^{ij}\) \)^{\1{\tq-1}}
        	= \sumn[n]{i}\( \frac{\tq-1}{\lam \tq}\((\alpha^{k+1})^i+(\beta^{k+1})^j - C^{ij}\) \)^{\1{\tq-1}}, \\ 
		\displaystyle r_k(\beta^k)^j + r_{k+1}(\beta^{k+1})^j = -\sum_{l\neq k,k+1}r_l(\beta^l)^j , 
		\qquad 1\leq j\leq n, 
	\end{array} \right. 
\end{align*}
only for $(\beta^k)^j$ and $(\beta^{k+1})^j$. 
Deleting the variable $(\beta^{k+1})^j$ by using the second relation, 
the solution $(\beta^k)^j$ can also be computed by the Newton method.


	\section{Geometrical perspective of the problems with a weakened assumption}

In this section, we consider a more general case.  
As in the previous section, we assume that a strictly convex function $\Phi:\Pnm\to\R$ has 
a smooth extension $\tPhi$ to $\Rnm$. 
In this section, 
we weaken the assumption that the derivative $S:\Rnm\to\R^{n\times m}$ of $\tPhi$ is surjective, 
and consider a situation where each $S^{ij}$ is not necessarily unbounded from above. 
The lack of surjectivity of $S$ implies that the optimal plan can be located on the boundary of $\Pnm$. 
As seen in the above, the mapping $S$ gives a correspondence between the primal and dual domains, 
described in mutually dual affine coordinates. 
Thus, a boundary of the dual domain
corresponds to the boundary of the primal domain $\Pnm$. 
Such a situation makes the problems difficult; however, it also can provide an advantage. 
It allows some masses of the optimal plan to be strictly zero, 
and thus, it can avoid a blurred image, 
which has been a problem of the original entropic regularization.

\subsection{Duality and subdifferential}
Due to the convex analysis, we obtain the picture as follows. 
Let $\hPhi:\R^{n\times m}\to\R\cup\{+\infty\}$ be 
the continuous extension of $\tPhi:\R_{++}^{n\times m}\to\R$, 
that is, 
\begin{align*}
	\hPhi(A) = \left\{ \begin{array}{l}
		\displaystyle \tPhi(A), \quad A\in \Rnm \\[1.5mm]
		\displaystyle \lim_{\tA\to A} \tPhi(\tA), \quad A\in\R^{n\times m}_{+}\setminus \Rnm \\[3mm]
		\displaystyle +\infty, \quad A\notin\R^{n\times m}_{+}
	\end{array}\right. .
\end{align*}
Then, the convex function $\hPhi$ is not smooth only on 
\[
	\del\R_+^{n\times m} := \Set{A\in\R^{n\times m} | \text{$A_{ij}=0$ for some $(i,j)$} }. 
\]
In contrast with the one-to-one correspondence 
\[
	\eta \longleftrightarrow \theta=\pd{\hPhi}{\eta}(\eta)
\]
on $\R_{++}^{n\times m}$, we make use of a one-to-many correspondence
\[
	\eta \longleftrightarrow \del\hPhi(A(\eta))
\]
on $\del\R_+^{n\times m}$.   
Here, $\del\hPhi(A)$ denotes the subdifferential of $\hPhi$ at $A\in\R^{n\times m}_+$, 
which is a convex subset of $\R^{n\times m}$ defined by
\[
	\del\hPhi(A) := \Set{ S\in\R^{n\times m} | 
		\hPhi(\tA) \geq \hPhi(A) + \<S,\tA -A\>, \any\tA\in\R_{++}^{n\times m} }. 
\]
By using this type of one-to-many correspondence, 
we construct a pseudo-surjective mapping from the primal domain to the dual domain.
The subdifferential on $\del\R_+^{n\times m}$ is given in detail by the next lemma. 

\begin{lemma}	\label{thm:subdifferential}
	Suppose that the extension of $\hPhi$ is finite and of $C^1$ on $\R_+^{n\times m}$. 
	Then, for $A\in\del\R_+^{n\times m}$, letting
	\[
		\Lambda_A := \Set{ (i,j) | 1\leq i\leq n, 1\leq j\leq m \st A_{ij}\neq0 },   
	\]
	the subdifferential of $\hPhi$ at $A$ is given by
	\begin{align*}
		\del\hPhi(A) = \left\{ S=(S^{ij}) \middle| 
		\begin{array}{l}
			S^{ij} = \pd{\hPhi}{A_{ij}}(A), \quad (i,j)\in\Lambda_A, \\
			S^{ij}\leq \pd{\hPhi}{A_{ij}}(A), \quad (i,j)\notin\Lambda_A 
		\end{array} \right\},  
	\end{align*}
	where 
	\[
		\pd{\hPhi}{A_{ij}}(A) = \lim_{\tA\to A, ~\tA\in\R^{n\times m}_{++}} \pd{\tPhi}{A_{ij}}(\tA)
	\]
\end{lemma}
\begin{proof}
	Let $S=(S^{ij})\in\R^{n\times m}$ satisfy 
	\begin{align*}
		S^{ij} = \pd{\hPhi}{A_{ij}}(A), ~ (i,j)\in\Lambda_A, \qquad
		S^{ij} \leq \pd{\hPhi}{A_{ij}}(A), ~ (i,j)\notin\Lambda_A. 
	\end{align*}
	Then, for any $\tA\in\R_{+}^{n\times m}$, it holds that
	\[
		(i,j)\in\Lambda_A \Longrightarrow \tA_{ij}-A_{ij}=\tA_{ij}\geq 0, 
	\]
	and thus, 
	\begin{align*}
		\hPhi(A) + \<S,\tA -A\>
		&= \hPhi(A) + \sum_{i,j} S^{ij}\(\tA_{ij}-A_{ij}\) \\
		&\leq \hPhi(A) + \sum_{i,j} \pd{\hPhi}{A_{ij}}(A)\(\tA_{ij}-A_{ij}\) \\
		&\leq \hPhi(\tA). 
	\end{align*}
	This implies that $S\in\del\hPhi(A)$, 
	where we used $(\p\hPhi/\p A)(A)\in\p\hPhi(A)$ to obtain the last inequality. 
	
	Conversely, let $S\in\del\hPhi(A)$. 
	Since the restriction of $\hPhi$ onto the affine subspace
	\[
		\R_{++}^{\Lambda_A} := \left\{ \tA=(\tA_{ij}) \middle| 
		\begin{array}{l}
			\tA_{ij} > 0, \quad (i,j)\in\Lambda_A, \\
			\tA_{ij} = 0, \quad (i,j)\notin\Lambda_A 
		\end{array} \right\} 
	\]
	is strictly convex and differentiable, 
	the coordinates of the subgradient of the restriction must be $(\p\hPhi/\p A_{ij})(A)$,  
	which leads
	\[
		(i,j)\in\Lambda_A \Longrightarrow S^{ij}=\pd{\hPhi}{A_{ij}}(A). 
	\]
	For $(k,l)\notin\Lambda_A$, we assume that $S^{kl}>(\del\hPhi/\del A_{kl})(A)$. 
	Letting 
	\[
		\tA = \left\{ \begin{array}{l} 
			\tA_{kl}=A_{kl}+\ep, \\ 
			\tA_{ij}=A_{ij}, \quad (i,j)\neq(k,l) 
		\end{array} \right. , 
	\]
	from the Taylor expansion
	\begin{align*}
		\tPhi(\tA) 
		&= \hPhi(A) + \sum_{i,j} \pd{\hPhi}{A_{ij}}(A)\(\tA_{ij}-A_{ij}\) + O(\|\tA-A\|^2) \\
		&= \hPhi(A) + \ep \pd{\hPhi}{A_{kl}}(A) + O(\ep^2),  
	\end{align*}
	we obtain, for sufficiently small $\ep>0$,
	\begin{align*}
		\tPhi(\tA) 
		= \hPhi(A) + \ep S^{kl}  
		< \hPhi(A) + \sum_{i,j} S^{ij}\(\tA_{ij}-A_{ij}\). 
	\end{align*}
	This contradicts to $S\in\p\hPhi(A)$, and hence, 
	\[
		(i,j)\notin\Lambda_A \Longrightarrow S^{ij}\leq\pd{\hPhi}{A_{ij}}(A). 
	\]
\end{proof}

\subsection{Quadratic regularization of optimal transport}
We introduce an algorithm computing the barycenter in a similar way to the previous section. 
However, the correspondence of $\eta$ and $\p\hPhi(A)$ is not one-to-one for $A\in\p\R^{n\times m}_+$.    
Thus, in order to construct a practical algorithm, 
we need to clarify that correspondence. 
Lemma \ref{thm:subdifferential} helps to find a concrete correspondence. 
In this subsection, we will introduce 
the optimal transport problem with quadratic regularization as an instance. 
It is also studied by Essid and Solomon \cite{EssidSolomon}; 
however, they assumed that the cost matrix $C$ is induced from a distance on a graph.  
In our framework, such an assumption is not required. 

We consider the convex function
\[
	\tPhi(A) = \<A,C\> + \frac{\lam}{2}\sum_{i,j}A_{ij}^2, \quad A\in\Rnm,   
\]
for a fixed $\lam>0$. 
Its continuous extension is given by
\begin{align*}
	\hPhi(A) = 
	\left\{\begin{array}{l}
		\displaystyle \<A,C\> + \frac{\lam}{2}\sum_{i,j}A_{ij}^2, \quad A\in\R^{n\times m}_+, \\[3mm]
		+\infty, \quad \ow
	\end{array}\right. . 
\end{align*}
Due to Lemma \ref{thm:subdifferential}, for $A\in\del\R_+^{n\times m}$, we obtain 
\begin{align*}
	\del\hPhi(A) = \left\{ S=(S^{ij}) \middle| 
	\begin{array}{l}
		S^{ij} = C^{ij}+\lam A_{ij}, \quad (i,j)\in\Lambda_A, \\
		S^{ij}\leq C^{ij}, \quad (i,j)\notin\Lambda_A 
	\end{array} \right\} , 
\end{align*}
and hence, given $S=(S^{ij})\in\R^{n\times m}$, 
the corresponding $A(S)\in\R_+^{n\times m}$ is presented explicitly by 
\begin{align*}
	A(S)_{ij} = \1{\lam}\( S^{ij} - C^{ij} \)^+, 
\end{align*}
where $x^+:=\max\{x,0\}$. 
Due to Lemma \ref{thm:dual}, there exists a subgradient $(\alpha,\beta)\in\R^n\times\R^m$ such that 
the optimal plan $P^*(p,q)$ is given as
\[
	P^*(p,q)_{ij} = \1{\lam}\( \alpha^i + \beta^j - C^{ij} \)^+. 
\]

For the minimization problem \eqref{eq:gen-prob}, 
the procedure (S-I) to obtain the $\nabla^*$-projection onto $M_{p,\cdot}$ is 
written as
\begin{equation}	\label{eq:S-I}
	\sum_j \1{\lam}\(\alpha^i+\beta^j-C^{ij}\)^+ = p_i.
\end{equation}
Solving this equation is implemented by Algorithm \ref{alg:S-I}. 
In the algorithm, the function $\textit{sort}$ returns 
the vector obtained by rearranging the entries of $\gamma$ 
so that $\gamma^{\downarrow}_1\geq\gamma^{\downarrow}_2\geq\cdots\geq\gamma^{\downarrow}_m$. 
This rearrangement is implemented, for example, by the merge sort, 
which costs $O(m\log m)$ time. 
The procedure (S-II) can be similarly implemented,   
and iterating these two procedures realizes a generalization of the Sinkhorn algorithm. 
The time required for (S-I) and (S-II) is 
$O(\max\{n,m\}^3)$ in the worst case because of the \textbf{for} loop in Algorithm \ref{alg:S-I};
however, from the observation in numerical simulations, 
it seems that this loop can be exited in a constant time. 
Thus, our algorithm, in practice, costs $O(\max\{n,m\}^2)$ time for each iteration, 
similarly to Cuturi's algorithm. 

\begin{algorithm}[htbp]
	\caption{Solver of Eq. \eqref{eq:S-I} for each $i$}                          
	\begin{algorithmic}	\label{alg:S-I}
		\REQUIRE $\lam>0, 1\leq i\leq n, \beta=(\beta^j), p=(p_i), C=(C^{ij})$ 
		\ENSURE $\sum_j \1{\lam}\(\alpha^i+\beta^j-C^{ij}\)^+ = p_i$
		\STATE $\gamma \Leftarrow (\beta^j-C^{ij})_j$
		\STATE $\gamma^{\downarrow} \Leftarrow \mathrm{sort}(\gamma)$ 
		\FOR{$J=1$ to $m$}
			\STATE $\alpha^i  \Leftarrow \(\lam p_i - \sum_{j=1}^J \gamma^{\downarrow}_j\)/ J$
			\IF{$\sum_j \1{\lam}\(\alpha^i+\beta^j-C^{ij}\)^+ = p_i$}
				\STATE \textbf{break}
			\ENDIF
		\ENDFOR
	\end{algorithmic}
\end{algorithm}

For the barycenter problem, 
the procedure (B-I) is presented in a similar form to \eqref{eq:S-I},  
which is solved by Algorithm \ref{alg:S-I} for each $k$ and $i$. 
On the other hand, the procedure (B-II) is hard to solve, 
and we make use of the alternative procedure (B-II)$_k$, 
which is written as the equation 
\begin{align}		\label{eq:B-IIk}
\begin{aligned}
	\sum_i \1{\lam}\( (\alpha^k)^i+(\beta^k)^j - C^{ij} \)^+ 
	&= \sum_i \1{\lam}\( (\alpha^{k+1})^i+(\beta^{k+1})^j - C^{ij} \)^+, \\
	\quad r_k (\beta^k)^j + r_{k+1} (\beta^{k+1})^j &= -\sum_{l\neq k,k+1} r_l (\beta^l)^j =: \Sigma^{k;j}, 
\end{aligned}
\end{align}
of $\beta^k,\beta^{k+1}\in\R^n$, 
where $(\alpha^1,\dots,\alpha^N)$ and $(\beta^1,\dots,\beta^{k-1},\beta^{k+2},\dots,\beta^N)$ are fixed. 
We can solve this equation by Algorithm \ref{alg:B-IIk}. 

In summary, 
the barycenter for the quadratic regularized Wasserstein distance is obtained by
applying Algorithm \ref{alg:S-I} to $p^k$ for $1\leq k\leq N$ 
and performing Algorithm \ref{alg:B-IIk} for $1\leq k\leq N-1$, iteratively. 
Since Algorithm \ref{alg:B-IIk}, in the worst case, costs $O(n^2)$
time for each $j$ and $k$, 
each iteration of the main loop may require the cost $O(Nn^3)$.  
However, also in this case, 
the \textbf{while} loop in Algorithm \ref{alg:B-IIk} can be exited with less cost. 
At the best performance, each iteration of our algorithm requires $O(Nn^2)$, 
which is the same order as Benamou \etal['s] algorithm does. 

\begin{algorithm}[H]
	\caption{Solver of Eq. \eqref{eq:B-IIk} for each $j$ and $k$}                          
	\begin{algorithmic}	\label{alg:B-IIk}
		\REQUIRE $\lam>0, 1\leq j\leq n, 1\leq k\leq N-1, C=(C^{ij}), (\alpha^k,\alpha^{k+1}), \Sigma^{k;j}$ 
		\ENSURE $(\beta^k)^j, (\beta^{k+1})^j$ solves Eq. \eqref{eq:B-IIk}
		\STATE $\gamma^1 \Leftarrow ((\alpha^k)^i-C^{ij})_i; 
			~\gamma^2 \Leftarrow ((\alpha^{k+1})^i-C^{ij})_i$
		\STATE $(\gamma^1)^{\downarrow} \Leftarrow \mathrm{sort}(\gamma^1); 
			~(\gamma^2)^{\downarrow} \Leftarrow \mathrm{sort}(\gamma^2)$
		\STATE $I_1 \Leftarrow 1; ~I_2 \Leftarrow 1$
		\WHILE{$I_1< n$ and $I_2< n$}
			\STATE $(\beta^k)^j \Leftarrow \1{r_{k+1} I_1+r_k I_2}\(I_2\Sigma^{k;j} 
				- r_{k+1}\(\sum_{i=1}^{I_1} (\gamma^1)^{\downarrow}_i 
				- \sum_{l=1}^{I_2} (\gamma^2)^{\downarrow}_l\)\)$
			\STATE $(\beta^{k+1})^j \Leftarrow \1{r_{k+1}}\(\Sigma^{k;j} - r_k(\beta^k)^j\)$
			\IF{$\sum_i \((\alpha^k)^i+(\beta^k)^j-C^{ij}\)^+
				= \sum_i \((\alpha^{k+1})^i+(\beta^{k+1})^j-C^{ij}\)^+$}
				\STATE \textbf{break}
			\ELSIF{{\small $\(\sum_i \((\alpha^k)^i+(\beta^k)^j-C^{ij}\)^+
				-\sum_i \((\alpha^{k+1})^i+(\beta^{k+1})^j-C^{ij}\)^+\)
				\cdot(\gamma^1)^{\downarrow}_{I_1+1} \leq 0$}}
				\STATE $I_1 \Leftarrow I_1+1$
			\ELSE
				\STATE $I_2 \Leftarrow I_2+1$
			\ENDIF
		\ENDWHILE
	\end{algorithmic}
\end{algorithm}

\subsection{Numerical Simulations}
We performed numerical simulations of solving the barycenter problems
with Cuturi's entropic regularization \eqref{eq:bary} 
and with quadratic regularization. 
We prepared two $32\times 32$ pixel grayscale images as extreme points $p^1,p^2\in\calP_{n-1}$, 
where $n=32\times 32=1024$. 
They are presented in Fig. \ref{fig1}. 

\begin{figure}[htbp]
\begin{center}
	\includegraphics[width=0.2\textwidth]{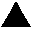}
	\hspace{2mm}
	\includegraphics[width=0.2\textwidth]{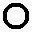}
	\caption{
		Two $32\times 32$ pixel grayscale images prepared as $p^1,p^2\in\calP_{n-1}$. 
		We regard each pixel as event $i\in\{1,2,\dots, n\}$
		and brightness of each pixel as a density on $i$, 
		that is, a white pixel has its density zero. 
	}
	\label{fig1} 
\end{center}
\end{figure}

\begin{figure}[htbp]
\begin{center}
    	\includegraphics[width=0.86\textwidth]{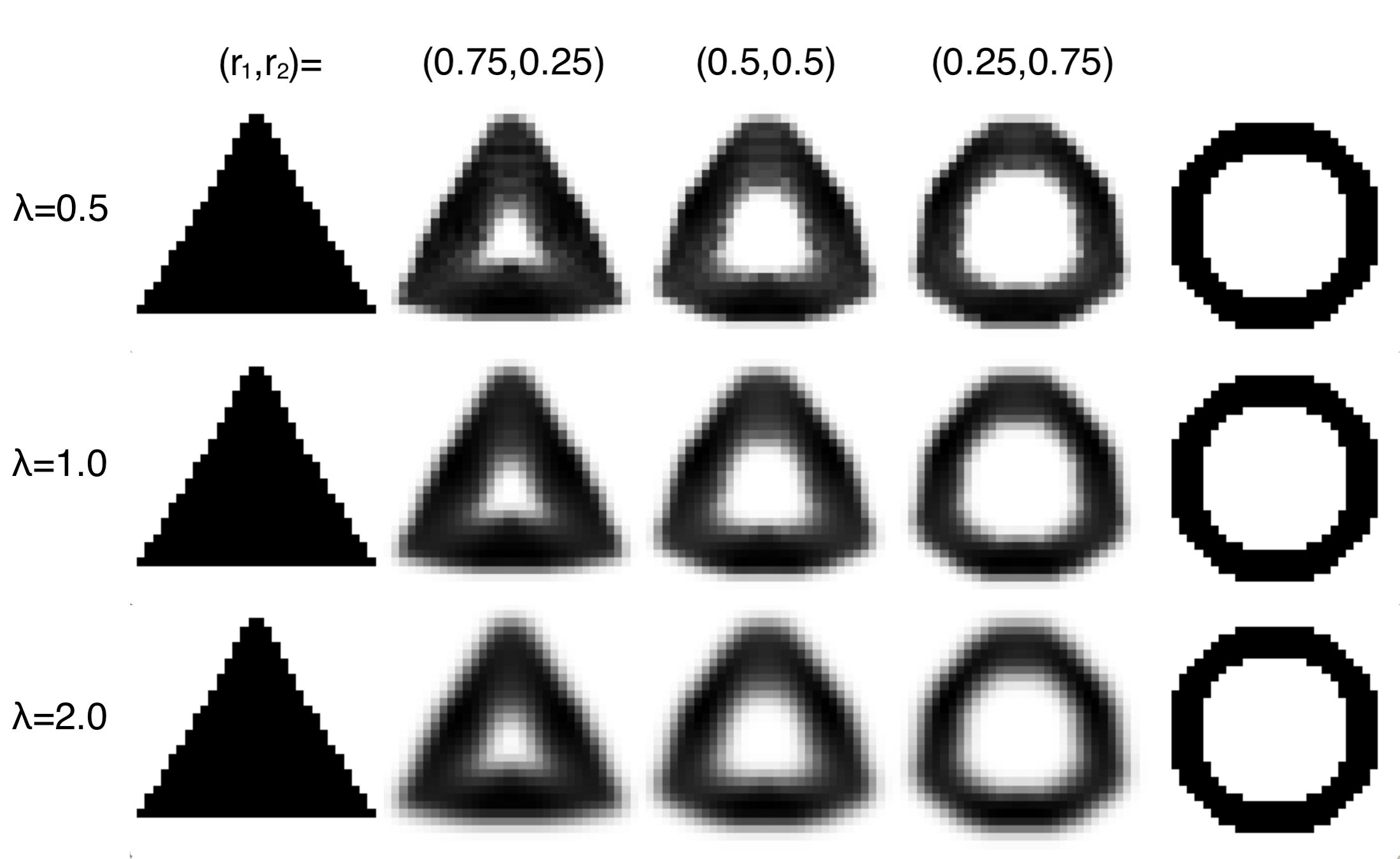}
	\caption{
		Wasserstein barycenters regularized by Shannon entropy 
		computed by Benamou \textit{et al.}'s algorithm. 
		The number of iteration is 3,000,  
		and the regularization constant $\lambda$ is 0.5 (top), 1.0 (middle), and 2.0 (bottom). 
		The ratio $(r_1,r_2)$ is set to be (0.75,0.25), (0.5,0.5), (0.25,0.75), for each simulation. 
	}
	\label{fig2} 
\end{center}
\end{figure}

\begin{figure}[htbp]
\begin{center}
    	\includegraphics[width=0.86\textwidth]{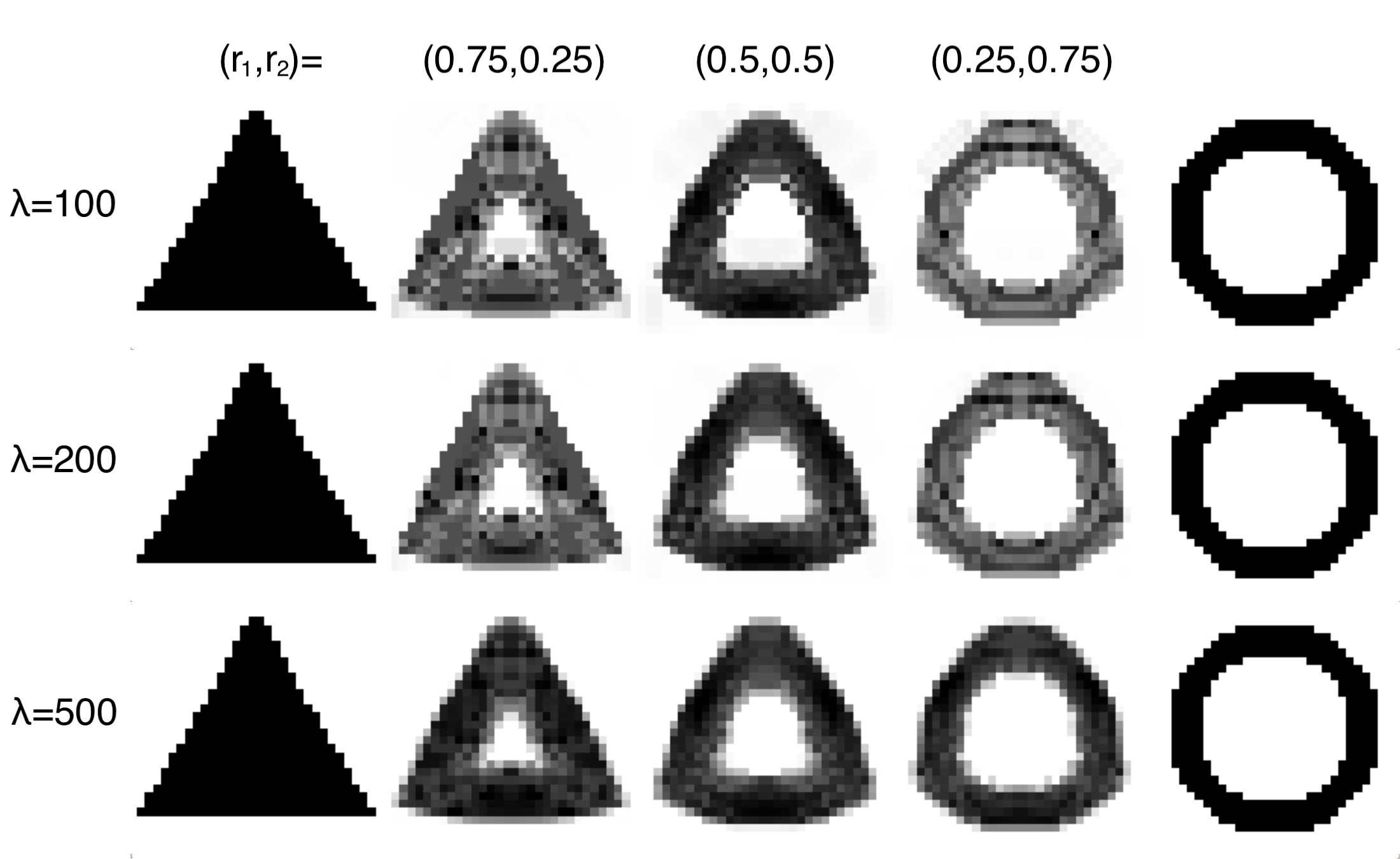}
	\caption{
		Wasserstein barycenters regularized by quadratic term
		computed by our algorithm. 
		The number of iteration is 3,000,  
		and the regularization constant $\lambda$ is 100 (top), 200 (middle), and 500 (bottom). 
		The ratio $(r_1,r_2)$ is set to be (0.75,0.25), (0.5,0.5), (0.25,0.75), for each simulation. 
	}
	\label{fig3} 
\end{center}
\end{figure}

We set the number of iteration to be 3,000, 
and the ratio $(r_1,r_2)$ of the \Frechet~mean to be (0.75,0.25), (0.5,0.5), and (0.25,0.75). 
Fig. \ref{fig2} presents the results of numerical simulations 
for Cuturi's entropic regularization solved by Benamou \etal['s] algorithm (Algorithm \ref{alg2}), 
and the regularization constant $\lambda$ is set to be 0.5, 1.0, and 2.0. 
Fig. \ref{fig3} does that of quadratic regularization solved by our algorithm
(Algorithm \ref{alg:S-I} and \ref{alg:B-IIk}), 
where $\lambda$ is set to be 100, 200, and 500. 

One can see that the barycenters with Cuturi's regularization are blurred 
compared to the ones with quadratic regularization. 
For small $\lambda$, such a blurring effect is suppressed; 
however, Benamou \etal['s] algorithm cannot perform well
for smaller $\lambda$ because of underflow.   
In fact, the values $K_{ij}=\exp(-C^{ij}/\lam)$ become too small for practical computation when $\lam$ is small. 
On the contrary, our algorithm uses only summation, and does not cause the underflow. 
In addition, one can check that the quadratic regularization yields fine solutions 
even when $\lambda$ is large.


	\section{Concluding remarks}

In the present paper, 
we generalized the entropic regularized optimal transport problem 
by Cuturi \cite{Cuturi13},  
and studied the minimization problem of a strictly convex smooth function $\Phi$ 
on the set $\Pnm$ of joint distributions under the constraint that marginal distributions are fixed. 
We clarified that the solution of the problem is represented in a specific form 
using a pair of dual affine coordinate systems (Theorem \ref{thm:onto}), 
and proposed an iterative method for obtaining the solution. 
We also studied the barycenter problem \cite{CuturiDoucet} 
from an information geometric point of view. 
We further showed numerically that 
our method works particularly efficient for $\Phi(P)=\<C,P\>+\frac{\lam}{2}\|P\|^2$.  

The framework treated in this paper is a maximal extension of Amari-Cuturi's one. 
Our framework subsumes some important problems, such as
the Tsallis entropic regularized transport problem \cite{MuzellecNPN}, 
which is represented in our framework with $\Phi(P)=\<C,P\>-\lam\calT_{\tq}(P)$, 
and the quadratic regularized one on a graph \cite{EssidSolomon}, 
which corresponds to $\Phi(P)=\<C,P\>+\frac{\lam}{2}\|P\|^2$ 
with $C$ being a metric matrix induced from a graph. 
Given the generality of our framework, 
it is expected that there exist many other applications. 
Finding another practical problem to which our framework can be applied is a future task. 

It was proved that, with our method, 
the Bregman divergence associated to the convex function $\Phi$
as measured from the optimal solution monotonically decreases;  
however, its convergence property is still unclear. 
For the Sinkhorn algorithm, a special case where $\Phi(P)=\<C,P\>-\lam\calH(P)$, 
Franklin and Lorenz \cite{FranklinLorenz} studied the convergence rate, 
and showed the exponentially fast convergence of the sequence generated from the algorithm 
with respect to the Hilbert metric. 
However, their analysis is specialized to the Sinkhorn algorithm, 
and it is difficult to extend their result to a generic case, 
since the Hilbert metric has, to the best of our knowledge, 
no relation with the information geometry. 
For a generic $\Phi$, evaluating the convergence rate of our method is an open problem.

\section*{Acknowledgements}
	The author would like to thank Professor Akio Fujiwara for his helpful guidance, discussions, and advice.



\section*{Appendix: Dual problem and proof of Theorem \ref{thm:happy}}	\label{app1}
\renewcommand{\thetheorem}{A.\arabic{theorem}}

In this section, we give a proof of Lemma \ref{thm:dual} and Theorem \ref{thm:happy}, 
and prove the existence of a dual solution (Lemma \ref{thm:dual-sol-exist}). 
Lemma \ref{thm:dual} follows from the Fenchel-Rockafellar duality, 
which is a prominent result in convex analysis. 
In order to prove Theorem \ref{thm:happy} analogously to Theorem \ref{thm:onto}, 
we prepare Lemma \ref{thm:1-homDual}, which provides
the dual problem for a 1-homogeneous $\Phi$. 

\begin{proposition}[Fenchel-Rockafellar duality]	\label{thm:FR-dual}
	Let $\Theta,\Xi:\R^n\to\R\cup\{+\infty\}$ be convex functions
	which are proper, i.e., the sets $\{\Theta(u)<+\infty\}$ and $\{\Xi(u)<+\infty\}$ are not empty. 
	Let $\Theta^*,\Xi^*$ be their Legendre transformations, respectively. 
	Then, the equality
	\[
		\inf_{u\in\R^n}\Set{\Theta(u)+\Xi(u)} = \sup_{A\in\R^n}\Set{-\Theta^*(-A)-\Xi^*(A)}
	\]
	holds. 
\end{proposition}

\begin{proof}[Proof of Lemma \ref{thm:dual}]
	Applying Proposition \ref{thm:FR-dual} to the convex functions
	\begin{align*}
		\Theta(u) = \tPhi^*(-u), \qquad
		\Xi(u) = \left\{ \begin{array}{l}
			-\<p,\alpha\>-\<q,\beta\> \quad \text{if $u=\alpha\oplus\beta$} \\
			+\infty \quad \ow
		\end{array}\right. ,
	\end{align*}
	since one can check that
	\begin{align*}
		\Theta^*(-A) &= \left\{ \begin{array}{l}
			\hPhi(A) \quad \text{if $A\in\R^{n\times m}_{+}$} \\
			+\infty \quad \ow
		\end{array} \right. , \\[2mm]
		\Xi^*(A) &= \left\{ \begin{array}{l}
			0 \quad \text{if $\sumn[m]{j}A_{ij}=p_i ~\sumn[n]{i}A_{ij}=q_j$} \\
			+\infty \quad \ow
		\end{array}\right. ,
	\end{align*}
	the conclusion follows,
	where $\hPhi$ denotes the continuous extension of $\tPhi$ onto $\R^{n\times m}_+$.  
\end{proof}

\begin{lemma}	\label{thm:dual-sol-exist}
	Let $\tPhi : \R^{n\times m}_{++} \to \R$ be a convex function 
	and $\tPhi^* : \R^{n\times m}\to \R\cup\{+\infty\}$ be its Legendre transform. 
	Then, there exists a solution $(\alpha^*,\beta^*)$ 
	of the optimization problem
	\begin{equation}	\label{eq:dual-sol-exist}
		\sup_{\alpha\in\R^n, \beta\in\R^m} \{\<p,\alpha\>+\<q,\beta\> -\tPhi^*(\alpha\oplus\beta)\}  
	\end{equation}
	for any $p\in\calP[n-1]$, $q\in\calP[m-1]$. 
\end{lemma}
\begin{proof}
	Let $\Psi$ be a function on $\R^n\times\R^m$ defined by
	\[
		\Psi(\alpha,\beta) = \tPhi^*(\alpha\oplus\beta).  
	\]
	Then, $\Psi$ becomes a convex function on the standard affine structure on $\R^n\times\R^m$. 
	The problem \eqref{eq:dual-sol-exist} is no other than the Legendre transformation 
	\[
		\Psi^*(p,q) 
		= \sup_{(\alpha,\beta)\in\R^n\times\R^m} \{\<(p,q),(\alpha,\beta)\> -\Psi(\alpha,\beta)\}, 
	\]
	and its supremum $(\alpha^*,\beta^*)$ is given by a subgradient of $\Psi^*$ at $(p,q)$. 
	Hence, unless $\Psi^*(p,q)=+\infty$, we can conclude that $(\alpha^*,\beta^*)$ exists. 
	Due to Lemma \ref{thm:dual}, we can see that
	\[
		\Psi^*(p,q) = \inf_{P\in\Pi(p,q)} \tPhi(P), 
	\]
	and it cannot be infinite, since $\tPhi$ is finite-valued on $\Rnm$. 
\end{proof}

Before we prove Theorem \ref{thm:happy}, we observe the dual problem for a 1-homogeneous function $\tPhi$. 
The next lemma is a direct consequence of Lemma \ref{thm:dual}, 
and we thus omit a proof. 

\begin{lemma}	\label{thm:1-homDual}
	Let $\Phi : \R^{n\times m}_{++} \to \R$ be a 1-homogeneous convex function. 
	For $p\in\calP[n-1]$, $q\in\calP[m-1]$, 
	\[
		\inf_{P\in\Pi(p,q)} \Phi(P) = \sup \Set{\<p,\alpha\>+\<q,\beta\> | 
			\begin{array}{l}
				\alpha\in\R^n, \beta\in\R^m, \\
				\Phi^*(\alpha\oplus\beta) = 0 
			\end{array} }. 
	\]
\end{lemma}

The next result is also important when treating a 1-homogeneous convex function. 
\begin{lemma}	\label{thm:score}
	Let $\tPhi:\Rnm\to\R$ be a 1-homogeneous convex function, 
	and $S:\Rnm\to\R^{n\times m}$ be its derivative. 	
	Then, 
	\[
		\tPhi(A) = \<A, S(A)\> \geq \<A,S(\tA)\>,  
	\]
	for any $A,\tA\in\Rnm$.  
\end{lemma}
\begin{proof}
	First, for any $\rho\neq 1$, since $S(A)\in\p\tPhi(A)$, we have
	\[
		H(\rho A) \geq \<\rho A - A, S(A)\> + H(A). 
	\]
	Then, since $\tPhi$ is 1-homogeneous, we can write
	\[
		(\rho-1)H(A) \geq (\rho-1)\<A, S(A)\>. 
	\]
	We can choose $\rho\neq 1$ arbitrarily, and hence, the first equality is shown. 
	
	Second, since $S(\tA)\in\p\tPhi(\tA)$, we have
	\[
		\tPhi(A) \geq \<A-\tA, S(\tA)\> + \tPhi(\tA)  
	\]
	for any $A\in\Rnm$. 
	Then, since $\tPhi(\tA)=\<\tA,S(\tA)\>$ as seen in the above, we finally obtain 
	\[
		\tPhi(A) \geq \<A,S(\tA)\>. 
	\]
\end{proof}

Now, we arrive at the proof of Theorem \ref{thm:happy}. 

\begin{proof}[Proof of Theorem \ref{thm:happy}]
	The proof proceeds in three steps. 
	First, let $P^*(p,q)\in\overline{\Pi(p,q)}$ be a solution of \eqref{eq:gen-prob}, 
	$(\alpha^*,\beta^*)$ be a dual solution, and we show that $\alpha^*\oplus\beta^* \in \partial\tPhi(P^*)$. 
	Second, we construct $P_*\in\Pnm$ such that $\nabla\tPhi(P_*)=\alpha^*\oplus\beta^*$. 
	Finally, we show that $P^*(p,q)$ and $P_*$ are equal, and thus it is located in $\Pi(p,q)$. 
	Here, we can assume that $\Phi(P^*(p,q))<\infty$ without loss of generality. 
	Thus, if necessary, let $\Phi(P^*(p,q)), S(P^*(p,q))$ denote the continuous extension of $\Phi, S$ 
	to $P^*(p,q)\in\overline{\Pi(p,q)}$, respectively.
	
	The existence of $P^*(p,q)$ follows from the compactness of $\overline{\Pi(p,q)}$. 
	Let $(\alpha^*,\beta^*)$ be a solution guaranteed in Lemma \ref{thm:dual-sol-exist}. 
	Due to the duality in Lemma \ref{thm:1-homDual}, 
	\begin{equation}	\label{eq:sub1}
		\Phi(P^*(p,q)) = \<p,\alpha^*\> + \<q,\beta^*\> 
		= \<P^*(p,q),\alpha^*\oplus\beta^*\> - \tPhi^*(\alpha^*\oplus\beta^*),   
	\end{equation}
	which implies that $\alpha^*\oplus\beta^*\in\p\tPhi(P^*(p,q))$. 
	
	Since $S:\Pnm\to\R^{n\times m}/\<\onevec{nm}\>$ is surjective by assumption, 
	there exist $c\in\R$ and $P_*\in\Pnm$ such that
	\begin{equation}	\label{eq:second-assert}
		S(P_*) = \pd{\tPhi}{A}(P_*) = \alpha^*\oplus\beta^* + c \, \onevec{nm} .  
	\end{equation}
	If we assume that $c>0$, from Lemma \ref{thm:score}, for $A\in\R^{n\times m}_+$, 
	\begin{align*}
		\tPhi(A) & \geq \<A, S(P_*)\> \\ 
			& = \<A, \alpha^*\oplus\beta^*+c \onevec{nm}\>. 
	\end{align*}
	Letting $A$ tend to $P^*(p,q)$, compared with \eqref{eq:sub1}, it yields $0 \geq c$, 
	which leads to a contradiction. 
	On the other hand, if we assume that $c<0$, since $\tPhi^*(\alpha^*\oplus\beta^*)=0$, 
	we have
	\begin{align*}
		\<P_*, \alpha^*\oplus\beta^*\> 
		&\leq \<P_*, S(P_*)\> \\
		&= \<P_*, \alpha^*\oplus\beta^* + c\, \onevec{nm}\> , 
	\end{align*}
	which also leads to a contradiction. 
	Hence, we obtain $c=0$ or $S(P_*) = \alpha^*\oplus\beta^*$
	from \eqref{eq:second-assert}. 
	
	Finally, 
	we show that $P^*(p,q)=P_*\in\Pi(p,q)$ by contradiction. 
	Otherwise, since $\Phi$ is strictly convex on $\Pnm$, for $t\in(0,1)$, 
	\begin{align*}
		\Phi(t P_* + (1-t)P^*(p,q)) &< t \Phi(P_*) + (1-t)\Phi(P^*(p,q)) \\
			&= \<t P_* - (1-t)P^*(p,q), \alpha^*\oplus\beta^*\> . 
	\end{align*}
	On the other hand, since $\alpha^*\oplus\beta^*$ is a subgradient of $\tPhi$ at $P^*(p,q)$, 
	\begin{align*}
		&\Phi(t P_* + (1-t)P^*(p,q)) \\
		&\geq \< \(t P_* - (1-t)P^*(p,q)\) - P^*(p,q), \alpha^*\oplus\beta^*\> + \Phi(P^*(p,q)) \\
		&= \< t P_* - (1-t)P^*(p,q), \alpha^*\oplus\beta^*\> . 
	\end{align*}
	Hence, we obtain a contradiction. This completes the proof. 
\end{proof}

%

\bibliographystyle{plain}      
\bibliography{ref}   

\end{document}